\documentclass[portuges,12pt,letter]{article}
\usepackage[centertags]{amsmath}
\usepackage{amsfonts}
\usepackage{newlfont}
\usepackage{amscd}
\usepackage{graphics}
\usepackage{epsfig}
\usepackage{indentfirst}
\usepackage{amsxtra}
\usepackage[latin1]{inputenc}
\usepackage{amssymb, amsmath}
\usepackage{amsthm}
\usepackage[mathscr]{eucal}

\newtheorem{thm}{Theorem}[section]

\newtheorem{remark}[thm]{Remark}

\author{Fabio Silva Botelho \\
Department of Mathematics \\
Federal University of Santa Catarina - UFSC \\
Florian\'{o}polis, SC - Brazil}
\date{}
\title{\bf  On duality principles for non-convex optimization with applications to superconductivity and some existence results for a model in non-linear elasticity }

\begin{document}
\maketitle

\abstract{This article develops   duality principles applicable to the Ginzburg-Landau system in superconductivity.
The main results are obtained through standard tools of convex analysis, functional analysis, calculus of variations and duality theory. In the second section, we present the general result for the case including a magnetic field and the respective magnetic potential in a local extremal context. Finally, in the last section we develop some global existence results for a model in elasticity.}

\section{Introduction}
In this work we present a theorem which represents a duality principle suitable for a large class of non-convex variational problems.

At this point we refer to the exceptionally important article "A contribution to contact problems for a class of solids and structures" by
W.R. Bielski and J.J. Telega, \cite{85},  published in 1985,  as the first one to successfully  apply and generalize the convex analysis approach to a model in non-convex and non-linear mechanics.

The present work is, in some sense, a kind of extension  of this previous work \cite{85} combined with a D.C. approach presented in \cite{12} and others such as \cite{2900}, which greatly influenced and
inspired my work and recent book \cite{120}.

First, we recall that about the year 1950 Ginzburg and Landau introduced a theory to model the super-conducting behavior of some types of materials below a critical temperature $T_c$,
which depends on the material in question. They postulated the free density energy may be written close to $T_c$ as
$$F_s(T)=F_n(T)+\frac{\hbar}{4m}\int_\Omega |\nabla \phi|^2_2 \;dx+\frac{\alpha(T)}{4}\int_\Omega |\phi|^4\;dx-\frac{\beta(T)}{2}\int_\Omega |\phi|^2\;dx,$$
where $\phi$ is a complex parameter, $F_n(T)$ and $F_s(T)$ are the normal and super-conducting free energy densities, respectively (see \cite{100}
 for details).
Here $\Omega \subset \mathbb{R}^3$ denotes the super-conducting sample with a boundary denoted by $\partial \Omega=\Gamma.$ The complex function $\phi \in W^{1,2}(\Omega; \mathbb{C})$ is intended to minimize
$F_s(T)$ for a fixed temperature $T$.

Denoting $\alpha(T)$ and $\beta(T)$ simply by $\alpha$ and $\beta$,  the corresponding Euler-Lagrange equations are given by:
\begin{equation} \left\{
\begin{array}{ll}
 -\frac{\hbar}{2m}\nabla^2 \phi+\alpha|\phi|^2\phi-\beta\phi=0, & \text{ in } \Omega
 \\ \\
 \frac{\partial {\phi}}{\partial \textbf{n}}=0, &\text{ on } \partial\Omega.\end{array} \right.\end{equation}
This last system of equations is well known as the Ginzburg-Landau (G-L) one.
In the physics literature is also well known the G-L energy in which a magnetic potential here denoted by $\textbf{A}$ is included.
The functional in question is given by:
\begin{eqnarray}\label{ar67}
J(\phi,\textbf{A})&=&\frac{1}{8\pi}\int_{\mathbb{R}^3} |\text{ curl }\textbf{A}-\textbf{B}_0|_2^2\;dx+\frac{\hbar^2}{4m}\int_\Omega \left|\nabla \phi-\frac{2 ie}{\hbar c}\textbf{A}\phi\right|^2_2\;dx \nonumber \\ &&+\frac{\alpha}{4}\int_\Omega|\phi|^4\;dx-\frac{\beta}{2}\int_\Omega |\phi|^2\;dx
\end{eqnarray}
Considering its minimization on the space $U$, where $$U= W^{1,2}(\Omega; \mathbb{C}) \times W^{1,2}(\mathbb{R}^3; \mathbb{R}^3),$$ through the physics notation the corresponding Euler-Lagrange equations are:
\begin{equation} \left\{
\begin{array}{ll}
 \frac{1}{2m}\left(-i\hbar\nabla -\frac{2e}{c}\mathbf{A}\right)^2 \phi+\alpha|\phi|^2\phi-\beta\phi=0, & \text{ in } \Omega
 \\ \\
 \left(i\hbar \nabla\phi+\frac{2e}{c}\textbf{A}\phi\right) \cdot \textbf{n}=0, &\text{ on } \partial\Omega,\end{array} \right.\end{equation}
and
\begin{equation} \left\{
\begin{array}{ll}
 \text{curl }(\text{curl } \textbf{A})= \text{ curl } \textbf{B}_0+\frac{4 \pi}{c} \tilde{J}, & \text{ in } \Omega
 \\ \\
 \text{curl }(\text{curl } \textbf{A})=\text{ curl }\textbf{B}_0, & \text{ in } \mathbb{R}^3\setminus \overline{\Omega},\end{array} \right.\end{equation}
 where $$\tilde{J}=-\frac{ie \hbar}{2m}\left(\phi^*\nabla \phi-\phi\nabla \phi^*\right)-\frac{2e^2}{mc}|\phi|^2 \textbf{A}.$$
 and $$\textbf{B}_0 \in L^2(\mathbb{R}^3; \mathbb{R}^3)$$ is a known applied magnetic field.

At this point,  we emphasize to denote generically
$$\langle g,h\rangle_{L^2}=\int_\Omega Re[g] Re[h]\;dx-\int_\Omega Im[g]Im[h]\;dx,$$
$\forall h,g \in L^2(\Omega;\mathbb{C})$, where $Re[a], Im[a]$ denote the real and imaginary parts of $a$,
$\forall a \in \mathbb{C},$ respectively.

Moreover, existence of a global solution for a similar problem has been proved at section \ref{section2} of this article and in \cite{500}.

Finally, for the subsequent theoretical results we assume a simplified atomic units context.

\section{ A global existence result for the full complex Ginzburg-Landau system}\label{section2}

In this section we present a global existence result for the complex Ginzburg-Landau system in superconductivity.

The main result is summarized by the following theorem.

\begin{thm} Let $\Omega,\Omega_1 \subset \mathbb{R}^3$ be open, bounded and connected sets whose the regular (Lipschistzian) boundaries are denoted by $\partial \Omega$ and $\partial \Omega_1$, respectively.

Assume $\overline{\Omega} \subset \Omega_1$.

Consider the functional $J:U \rightarrow \mathbb{R}$ defined by
\begin{eqnarray}
J(\phi,\mathbf{A})&=& \frac{\gamma}{2}\int_\Omega |\nabla \phi-i \rho \mathbf{A} \phi|^2\;dx \nonumber \\
&&+\frac{\alpha}{2}\int_\Omega (|\phi|^2-\beta)^2\;dx+K_0\|\text{ curl }\mathbf{A}-\mathbf{B}_0\|_{2,\Omega_1}^2,
\end{eqnarray}
where $\alpha,\beta,\gamma,\rho,K_0$ are positive constants, $i$ is the imaginary unit,
$$U=U_1 \times U_2,$$
$$U_1=W^{1,2}(\Omega;\mathbb{C})$$ and
$$U_2=W^{1,2}(\Omega_1;\mathbb{R}^3).$$
Also $\mathbf{B}_0 \in W^{1,2}(\Omega_1;\mathbb{R}^3)$ denotes an external magnetic field.

Under such hypotheses, there exists $(\phi_0,\mathbf{A}_0) \in U$ such that
$$J(\phi_0,\mathbf{A}_0)=\min_{(\phi,\mathbf{A}) \in U} J(\phi,\mathbf{A}).$$
\end{thm}
\begin{proof}
Denote $$\alpha_1=\inf_{(\phi,\mathbf{A}) \in U}  J(\phi,\mathbf{A}).$$

From the results in \cite{500}, fixing the gauge of London, we may find a minimizing sequence for $J$ such that
$$\alpha_1 \leq J(\phi_n,\mathbf{A}_n) < \alpha_1+\frac{1}{n},$$
$$\text{ div } \mathbf{A}_n=0, \text{ in } \Omega_1,$$
and $$\mathbf{A}_n \cdot \mathbf{n}=0, \text{ on } \partial \Omega_1, \; \forall n \in \mathbb{N}.$$

From $\text{ div } \mathbf{A}_n=0$ we get
$$\text{ curl curl }\mathbf{A}_n=-\nabla^2 \mathbf{A}_n,$$ so that from this and $$\mathbf{A}_n\cdot \mathbf{n} =0 \text{ on }
\partial \Omega_1,$$ we have that
$$\int_{\Omega_1} \text{ curl }\mathbf{A} \cdot \text{ curl }\mathbf{A}\;dx=\sum_{k=1}^3\int_{\Omega_1}\nabla (A_k)_n \nabla (A_k)_n\;dx,\; \forall n \in \mathbb{N}.$$

From the expression of $J$, for such a minimizing sequence, there exists a real $K_1>0$ such that
$$\|\nabla \mathbf{A}_n\|_{2,\Omega_1}=\|\text{ curl }\mathbf{A}_n\|_{2,\Omega_1} \leq K_1,\; \forall n \in \mathbb{N}.$$

From this and the boundary conditions in question, there exists a real $K_2>0$ such that
$$\|\mathbf{A}_n\|_{1,2,\Omega_1} \leq K_2,\; \forall n \in \mathbb{N}.$$

Thus from the Rellich-Kondrachov theorem, there exists $\mathbf{A}_0 \in U_2$ such that, up to a not relabeled
subsequence $$\mathbf{A}_n \rightharpoonup \mathbf{A}_0, \text{ weakly in } W^{1,2}(\Omega_1,\mathbb{R}^3),$$
and
$$\mathbf{A}_n \rightarrow \mathbf{A}_0, \text{ strongly in } L^2(\Omega_1,\mathbb{R}^3)\text{ and } L^4(\Omega_1;\mathbb{R}^3).$$

From this and the generalized H\"{o}lder inequality, we obtain
\begin{eqnarray}
\alpha_1+\frac{1}{n} > J(\phi_n,\mathbf{A}_n) &\geq& \frac{\gamma}{2}\int_\Omega |\nabla \phi_n|^2\;dx
\nonumber \\ &&-K_3\|\mathbf{A}_n\|_{4,\Omega}
\|\nabla \phi_n\|_{2,\Omega}\|\phi_n\|_{4,\Omega}
+\frac{\gamma}{2}\int_\Omega |\mathbf{A}_n|^2|\phi_n|^2\;dx \nonumber \\ &&
+\frac{\alpha}{2}\int_\Omega |\phi_n|^4\;dx-\alpha\beta \int_{\Omega}|\phi_n|^2\;dx \nonumber \\ &&
K_0 \|\text{ curl } \mathbf{A}_n-\mathbf{B}_0\|_{2,\Omega_1}^2 +K_5 \nonumber \\ &\geq&
\frac{\gamma}{2}\int_\Omega |\nabla \phi_n|^2\;dx
\nonumber \\ &&-K_4
\|\nabla \phi_n\|_{2,\Omega}\|\phi_n\|_{4,\Omega}
+\frac{\gamma}{2}\int_\Omega |\mathbf{A}_n|^2|\phi_n|^2\;dx \nonumber \\ &&
+\frac{\alpha}{2}\int_\Omega |\phi_n|^4\;dx-\alpha \beta \int_{\Omega}|\phi_n|^2\;dx \nonumber \\ &&
K_0\|\text{ curl } \mathbf{A}_n-\mathbf{B}_0\|_{2,\Omega_1}^2 +K_5, \forall n \in \mathbb{N}.
\end{eqnarray}

From this, there exists $K_6>0$ such that
$$\|\nabla \phi_n\|_{2,\Omega} \leq K_6,$$
and
$$\|\phi_n\|_{4,\Omega}\leq K_6,\; \forall n \in \mathbb{N}.$$

Thus, from the Rellich-Kondrachov theorem, there exists $\phi_0 \in U_1$ such that, up to a not relabeled subsequence,
$$\phi_n \rightharpoonup \phi_0, \text{ weakly in } W^{1,2}(\Omega;\mathbb{C}),$$
and
$$\phi_n \rightarrow \phi_0, \text{ strongly in } L^2(\Omega;\mathbb{C}) \text{ and } L^4(\Omega;\mathbb{C}).$$

From such results, we may  obtain
$$\nabla \phi_n -i\rho \mathbf{A}_n \phi_n \rightharpoonup \nabla \phi_0-i \rho \mathbf{A}_0\phi_0,
\text{ weakly in } L^2(\Omega;\mathbb{C}).$$
Therefore, from this, the convexity of $J$ in $v=\nabla \phi-i\rho\mathbf{A}\phi$ and $w=\text{ curl } \mathbf{A},$ and from
$$\phi_n \rightarrow \phi_0, \text{ strongly in } L^2(\Omega;\mathbb{C}) \text{ and } L^4(\Omega;\mathbb{C})$$
and $$\text{ curl }\mathbf{A}_n \rightharpoonup \text{ curl }\mathbf{A}_0, \text{ weakly in } L^2(\Omega_1,\mathbb{R}^3),$$
we get
$$\alpha_1=\inf_{(\phi,\mathbf{A}) \in U} J(\phi,\mathbf{A})=\liminf_{n \rightarrow \infty}J(\phi_n,\mathbf{A}_n)
\geq J(\phi_0,\mathbf{A}_0).$$

The proof is complete.
\end{proof}

\section{A brief initial description of our proposal for duality}
Let $\Omega \subset \mathbb{R}^3$ be an open, bounded and connected set with a regular boundary denoted by $\partial \Omega$. Let $U=W_0^{1,2}(\Omega)$ and let $J:U \rightarrow \mathbb{R}$ be a functional defined by

$$J(u)=G_1( u)+F_1(u)-\langle u,f \rangle_{L^2},$$

where

$$G_1(u)=\frac{\gamma}{2}\int_\Omega \nabla u \cdot \nabla u\;dx,$$
and
$$F_1(u)=\frac{\alpha}{2}\int_\Omega (u^2-\beta)^2\;dx,$$
where $\alpha,\beta$ and $\gamma$ are positive real constants and $f \in L^2(\Omega)$.

Observe that there exists $\eta \in \mathbb{R}$ such that

$$\eta=\inf_{u \in U} J(u)=J(u_0),$$ for some $u_0 \in U.$
(We recall that the existence of a global minimizer may be proven by the direct method of the calculus of variations).

In our approach, we combine the ideas of J.J. Telega \cite{2900,85} generalizing the approach in  Ekeland and Temam \cite{[6]} for establishing the dual functionals
through the Legendre transform definition, with a D.C. approach for non-convex optimization developed by J.F. Toland, \cite{12}.

At this point we would define the functionals $F:U \rightarrow \mathbb{R}$
and $G:U \rightarrow \mathbb{R}$, where

$$F(u)=\frac{\gamma}{2}\int_\Omega \nabla u \cdot \nabla u\;dx+\frac{K}{2}\int_\Omega u^2\;dx,$$

and

$$G(u,v)=-\frac{\alpha}{2}\int_\Omega (u^2-\beta+v)^2\;dx+\frac{K}{2} \int_\Omega u^2\;dx+\langle u,f \rangle_{L^2},$$
so that \begin{eqnarray}J(u)&=&F(u)-G(u,0) \nonumber \\ &=&
\frac{\gamma}{2}\int_\Omega \nabla u \cdot \nabla u\;dx+\frac{\alpha}{2}\int_\Omega (u^2-\beta)^2\;dx-\langle u,f \rangle_{L^2}, \end{eqnarray}
 where such a functional, for a large $K>0$,  is represented as a difference of two convex functionals
in a large domain region proportional to $K>0$.

The second step is to define the corresponding dual functionals $F^*$, $G^*$, where

\begin{eqnarray}
F^*(v^*_1)&=&\sup_{u \in U}\{ \langle u,v_1^* \rangle_{L^2}-F(u)\}
\nonumber \\ &=& \frac{1}{2}\int_\Omega \frac{(v_1^*)^2}{K-\gamma\nabla^2}\;dx
\end{eqnarray}

and

\begin{eqnarray}G^*(v_1^*,v_0^*)&=&\sup_{ u \in U}\inf_{v \in L^2} \{ \langle u,v_1^* \rangle_{L^2}-\langle v,v_0^* \rangle_{L^2} -G(u,v)\}
\nonumber \\ &=& \sup_{u \in U} \inf_{v \in L^2} \{ \langle u,v_1^* \rangle_{L^2}-\langle v,v_0^* \rangle_{L^2} +\frac{\alpha}{2}\int_\Omega (u^2-\beta+v)^2\;dx \nonumber \\ &&-\frac{K}{2}\int_\Omega u^2\;dx-\langle u,f \rangle_{L^2}\}
\nonumber \\ &=& \sup_{u \in U} \inf_{w \in L^2} \{ \langle u,v_1^* \rangle_{L^2}-\langle w-u^2+\beta,v_0^* \rangle_{L^2} +\frac{\alpha}{2}\int_\Omega w^2\;dx \nonumber \\ &&-\frac{K}{2}\int_\Omega u^2\;dx-\langle u,f \rangle_{L^2}\} \nonumber \\ &=&
\sup_{u \in U} \inf_{w \in L^2} \{ \langle u,v_1^* \rangle_{L^2}-\langle w,v_0^* \rangle_{L^2} +\frac{\alpha}{2}\int_\Omega w^2\;dx-\langle u^2,v_0^*\rangle_{L^2} \nonumber \\ && -\frac{K}{2}\int_\Omega u^2\;dx -\langle u,f \rangle_{L^2}\nonumber \\
&&-\beta \int_\Omega v_0^*\;dx\} \nonumber \\ &=&
-\frac{1}{2}\int_\Omega \frac{(v_1^*-f)^2}{2v_0^*-K}\;dx-\frac{1}{2\alpha}\int_\Omega (v_0^*)^2\;dx\nonumber \\
&&-\beta \int_\Omega v_0^*\;dx,\end{eqnarray}
if $-2v_0^*+K>0$ in $\overline{\Omega}.$

Defining $$E=\{v_0^* \in C(\overline{\Omega}) \text{ such that } -2v_0^*+K>K/2, \text{ in } \overline{\Omega}\},$$ where $K>0$ is such that $$\frac{1}{\alpha}> \frac{32K_2^2}{K^3},$$ for some appropriate $K_2>0$ to be specified, it may be proven that for $K>0$ sufficiently large,

$$\inf_{v_1^* \in L^2}\sup_{v_0^* \in E}\{ -F^*(v_1^*)+G^*(v_1^*,v_0^*)\} \geq \inf_{u \in U} J(u).$$

Equality concerning this last result may be obtained in a local extremal context and, under appropriate optimality conditions to be specified, also for global optimization.

At this point we highlight the maximization in $v_0^*$ with the restriction $v_0^* \in E$ does not demand a Lagrange
multiplier, since for the value of $K>0$ specified the restriction is not active.

We emphasize this approach is original and substantially different from all those, of other authors, so far known.

Finally, for a more general model, in the next section we formally prove that a critical point for the primal formulation necessarily corresponds to a critical point of the dual formulation. The reciprocal may be also proven.

\section{The duality principle  for a local extremal context}
In this section we state and prove the concerning duality principle. We recall the existence of a global minimizer
for the related functional has been proven at section \ref{section2} of this article. At this point it is also worth mentioning an extensive study on duality theory and applications for such and similar models is developed in \cite{120}.

\begin{thm} Let $\Omega,\Omega_1 \subset \mathbb{R}^3$ be open, bounded and connected sets with regular
(Lipschitzian) boundaries denoted by $\partial \Omega$ and $\partial \Omega_1$ respectively.

Assume $\Omega_1$ is convex and $\overline{\Omega} \subset \Omega_1$. Consider the functional $J:U \rightarrow \mathbb{R}$ where
\begin{eqnarray}
J(\phi,\mathbf{A})&=& \frac{\gamma}{2}\int_\Omega |\nabla \phi -i \rho \mathbf{A}\phi|^2\;dx \nonumber \\
&& +\frac{\alpha}{2} \int_\Omega (|\phi|^2 - \beta)^2\;dx +\frac{1}{8\pi}\|\text{ curl } \mathbf{A}-\mathbf{B}_0\|_{0, \Omega_1}^2
\end{eqnarray}
where $\alpha,\beta,\gamma, \rho$ are positive real constants, $i$ is the imaginary unit and
$$U=U_1 \times U_2,$$
$$U_1=C^1(\overline{\Omega};\mathbb{C}),\;\; U_2=C^1(\overline{\Omega_1};\mathbb{R}^3),$$
both with the norm $\|\cdot\|_{1,\infty}.$

Moreover, $$\phi :\Omega \rightarrow \mathbb{C}$$ is the order parameter,
$$\mathbf{A}:\Omega_1 \rightarrow \mathbb{R}^3$$ is the magnetic potential and $\mathbf{B}_0 \in C^{1}(\overline{\Omega_1},\mathbb{R}^3)$ is  an external magnetic field.

Defining, $$B_2=\{\mathbf{A} \in C^1(\overline{\Omega_1}; \mathbb{R}^3)\;:\; \text{ div } \mathbf{A}=0 \text{ in } \Omega_1, \; \mathbf{A} \cdot \mathbf{n}=0, \text{ on } \partial \Omega_1\},$$
where $\mathbf{n}$ denotes the outward normal to $\partial \Omega_1,$
suppose $(\phi_0, \mathbf{A}_0) \in C^1(\overline{\Omega} ;\mathbb{C}) \times B_2$ is such that
$$\|\phi_0\|_\infty \leq K_2,$$ for some appropriate $K_2>0$,
$$\delta J(\phi_0,\mathbf{A}_0)=\mathbf{0}$$ and
$$\delta^2 J(\phi_0,\mathbf{A}_0)> \mathbf{0}.$$

Denoting also generically $$(\nabla-i\rho \mathbf{A})^*(\nabla-i\rho\mathbf{A})=|\nabla-i\rho\mathbf{A}|^2,$$
define $F:U \rightarrow \mathbb{R}$ by
$$F(\phi,\mathbf{A})=\frac{\gamma}{2}\int_\Omega |\nabla \phi -i \rho \mathbf{A}\phi|^2\;dx+\frac{K}{2}\int_\Omega |\phi|^2\;dx,$$

$G:U \times C(\Omega) \rightarrow \mathbb{R}$ by
\begin{eqnarray}G(\phi,\mathbf{A},v)&=&-\frac{\alpha}{2} \int_\Omega (|\phi|^2 - \beta+v)^2\;dx -\frac{1}{8\pi}\|\text{ curl } \mathbf{A}-\mathbf{B}_0\|_{0, \Omega_1}^2+\frac{K}{2}\int_\Omega |\phi|^2\;dx
,\end{eqnarray}
\begin{eqnarray}
F^*(v_1^*,\mathbf{A})&=&\sup_{\phi \in U_1} \{\langle \phi,v_1^*\rangle_{L^2}-F(\phi,\mathbf{A})\}
\nonumber\\ &=&\frac{1}{2} \int_\Omega \frac{(v_1^*)^2}{(\gamma|\nabla-i\rho \mathbf{A}|^2+K)}\;dx,
\end{eqnarray}
\begin{eqnarray}
\hat{G}^*(v_1^*,v_0^*,\mathbf{A})&=&
\sup_{\phi \in U_1} \inf_{v \in C(\Omega)}\{\langle \phi,v_1^*\rangle_{L^2}-\langle v,v_0^*\rangle_{L^2}-G(\phi,\mathbf{A},v)\}
\nonumber \\ &=&-\frac{1}{2}\int_\Omega \frac{(v_1^*)^2}{2v_0^*-K}\;dx \nonumber \\ && -\frac{1}{2\alpha}\int_\Omega (v_0^*)^2\;dx-\beta \int_\Omega v_0^*\;dx \nonumber \\ &&+\frac{1}{8\pi}
\|\text{ curl }\mathbf{A}-\mathbf{B}_0\|_{0,\Omega_1}^2,\end{eqnarray}
if $$-2v_0^*+K>0 \text{ in } \overline{\Omega},$$
and
\begin{eqnarray}
J^*(v_1^*,v_0^*,\mathbf{A})&=& -F^*(v_1^*,\mathbf{A})+\hat{G}^*(v_1^*,v_0^*,\mathbf{A})
\nonumber \\ &=&-\frac{1}{2} \int_\Omega \frac{(v_1^*)^2}{(\gamma|\nabla-i\rho \mathbf{A}|^2+K)}\;dx
\nonumber \\ && -\frac{1}{2}\int_\Omega \frac{(v_1^*)^2}{2v_0^*-K}\;dx \nonumber \\ &&
-\frac{1}{2\alpha}\int_\Omega (v_0^*)^2\;dx-\beta \int_\Omega v_0^*\;dx \nonumber \\ &&+\frac{1}{8\pi}
\|\text{ curl }\mathbf{A}-\mathbf{B}_0\|_{0,\Omega_1}^2.\end{eqnarray}

Furthermore, define $$\hat{v}_0^*=\alpha(|\phi_0|^2 -\beta),$$

$$\hat{v}_1^*=(2\hat{v}_0^*-K)\phi_0,$$
and $$E=\{v_0^* \in C(\overline{\Omega})\;:\; -2v_0^*+K>K/2, \text{ in } \overline{\Omega}\}.$$
Under such hypotheses and assuming also $$\hat{v}_0^* \in E,$$ we have
$$\delta J^*(\hat{v}_1^*,\hat{v}_0^*,\mathbf{A}_0)=\mathbf{0}.$$
$$J(\phi_0,\mathbf{A}_0)=J^*(\hat{v}_1^*,\hat{v}_0^*,\mathbf{A}_0).$$
Moreover, defining $$J_1^*(v_1^*, \mathbf{A})=\sup_{v_0^* \in E}J^*(v_1^*,v_0^*,\mathbf{A}),$$  for $K>0$ such that $$\frac{1}{\alpha}> \frac{8 K_2^2}{K}$$ and
sufficiently large, we have $$\delta J^*_1(\hat{v}_1^*,\mathbf{A}_0)=\mathbf{0},$$
$$\frac{\partial^2 J^*_1(\hat{v}_1^*,\mathbf{A}_0)}{\partial (v_1^*)^2}> \mathbf{0}$$ so that there exist $r,r_1>0$ such that

\begin{eqnarray}
J(\phi_0,\mathbf{A}_0)&=&\min_{(\phi,\mathbf{A}) \in B_r(\phi_0,\mathbf{A}_0)}J(\phi,\mathbf{A})\nonumber \\ &=&\inf_{v_1^*\in B_{r_1}(\hat{v}_1^*)} J^*_1(v_1^*,\mathbf{A}_0)
\nonumber \\ &=& J^*_1(\hat{v}_1^*,\mathbf{A}_0) \nonumber \\ &=&
\inf_{v_1^* \in B_{r_1}(\hat{v}_1^*)}\left\{ \sup_{v_0^* \in E} J^*(v_1^*,v_0^*,\mathbf{A}_0)\right\}
\nonumber \\ &=& J^*(\hat{v}_1^*,\hat{v}_0^*,\mathbf{A}_0).\end{eqnarray}

\end{thm}
\begin{proof} We start by proving that
$$\delta J^*(\hat{v}_1^*,\hat{v}_0^*,\mathbf{A}_0)=\mathbf{0}.$$

Observe that from $$\frac{\partial J(\phi_0,\mathbf{A}_0)}{\partial \phi}=0,$$ and $$\frac{\partial J(\phi_0,\mathbf{A}_0)}{\partial \mathbf{A}}=0,$$ we have
\begin{equation}\label{eq10} \left\{
\begin{array}{ll}
 \gamma\left|\nabla -i \rho\mathbf{A}_0\right|^2 \phi_0+2\alpha(|\phi_0|^2-\beta) \phi_0=0, & \text{ in } \Omega
 \\ \\
 \left(\nabla\phi_0-i \rho\textbf{A}_0\phi_0\right) \cdot \textbf{n}=0, &\text{ on } \partial\Omega,\end{array} \right.\end{equation}
and
\begin{equation} \left\{
\begin{array}{ll}
 \text{curl }(\text{curl } \textbf{A}_0)= \text{ curl } \textbf{B}_0+4\pi \tilde{J}_0, & \text{ in } \Omega
 \\ \\
 \text{curl }(\text{curl } \textbf{A}_0)=\text{ curl }\textbf{B}_0, & \text{ in } \Omega_1 \setminus \overline{\Omega},\end{array} \right.\end{equation}
 where $$\tilde{J}=-2i \gamma \rho Im\left[(\phi_0^*\nabla \phi_0)\right]-\gamma \rho^2|\phi_0|^2 \textbf{A}_0.$$

 Observe also that \begin{eqnarray}\frac{\partial J^*(\hat{v}_1^*, \hat{v}_0^*,\mathbf{A}_0)}{\partial v_0^*}
 &=& \frac{(\hat{v}_1^*)^2}{(2\hat{v}_0^*-K)^2}-\frac{\hat{v}_0^*}{\alpha}-\beta \nonumber \\ &=&
 |\phi_0|^2-\frac{\hat{v}_0^*}{\alpha}-\beta=0.\end{eqnarray}

 Summarizing, we have got
 $$\frac{\partial J^*(\hat{v}_1^*, \hat{v}_0^*,\mathbf{A}_0)}{\partial v_0^*}=0.$$

Moreover, from the first line in equation (\ref{eq10}), we obtain
$$\gamma\left|\nabla -i \rho\mathbf{A}_0\right|^2 \phi_0+K\phi_0+2\alpha(|\phi_0|^2-\beta) \phi_0-K\phi_0=0,  \text{ in } \Omega,$$
so that \begin{eqnarray}\hat{v}_1^*&=&(2\hat{v}_0^*-K)\phi_0
\nonumber \\ &=&2\alpha(|\phi_0|^2-\beta) \phi_0-K\phi_0\nonumber \\ &=&
-\gamma\left|\nabla -i \rho\mathbf{A}_0\right|^2 \phi_0-K\phi_0.\end{eqnarray}

Hence, $$\phi_0=\frac{(\hat{v}_1^*)}{2\hat{v}_0^*-K}=-\frac{\hat{v}_1^*}{ \gamma|\nabla-i\rho\mathbf{A}_0|^2+K},$$
and thus,
 \begin{eqnarray}\frac{\partial J^*(\hat{v}_1^*, \hat{v}_0^*,\mathbf{A}_0)}{\partial v_1^*}
 &=& -\frac{(\hat{v}_1^*)}{2\hat{v}_0^*-K}-\frac{\hat{v}_1^*}{ \gamma|\nabla-i\rho\mathbf{A}_0|^2+K}
 \nonumber \\ &=& -\phi_0+\phi_0=0,\end{eqnarray}
Also,
denoting \begin{eqnarray}H_1&=& \frac{\partial [\gamma|\nabla-i\rho\mathbf{A}_0|^2]}{\partial \mathbf{A}}\left[\frac{1}{2}\frac{(\hat{v}_1^*)^2}{ (\gamma|\nabla-i\rho\mathbf{A}_0)|^2+K)^2}\right]+\frac{1}{4\pi}\left\{\text{curl }(\text{curl } \textbf{A}_0)- \text{ curl } \textbf{B}_0\right\}
 \nonumber \\ &=&\frac{\partial [\gamma|\nabla-i\rho\mathbf{A}_0|^2]}{\partial \mathbf{A}}\left[\frac{|\phi_0|^2}{2}\right]+\frac{1}{4\pi}\left\{\text{curl }(\text{curl } \textbf{A}_0)- \text{ curl } \textbf{B}_0\right\}
 \nonumber \\ &=& \frac{1}{4\pi}\left\{\text{curl }(\text{curl } \textbf{A}_0)- \text{ curl } \textbf{B}_0\right\}- \tilde{J}_0, \end{eqnarray}

 and
  $$H_2=\frac{1}{4\pi}\left\{\text{curl }(\text{curl } \textbf{A}_0)- \text{ curl } \textbf{B}_0\right\},$$

 we get
 \begin{equation}
\frac{\partial J^*(\hat{v}_1^*, \hat{v}_0^*,\mathbf{A}_0)}{\partial \mathbf{A}}=\left \{
\begin{array}{ll}
 H_1 &  \text{ in } \Omega,
 \\
 H_2, & \text{ in }\Omega_1\setminus \overline{\Omega}.
  \end{array} \right.\end{equation}



Summarizing,

$$\frac{\partial J^*(\hat{v}_1^*, \hat{v}_0^*,\mathbf{A}_0)}{\partial \mathbf{A}}=\mathbf{0}.$$

Such last results may be denoted by
$$\delta J^*(\hat{v}_1^*,\hat{v}_0^*,\mathbf{A}_0)=\mathbf{0}.$$

Recall now that $$J_1^*(v_1^*,\mathbf{A})=\sup_{v_0^* \in E} J^*(v_1^*,v_0^*, \mathbf{A}).$$

Thus, \begin{eqnarray} \frac{\partial J^*_1(\hat{v}_1^*,\mathbf{A}_0)}{\partial v_1^*} &=&
\frac{\partial J^*(\hat{v}_1^*,\hat{v}_0^*,\mathbf{A}_0)}{\partial v_1^*} \nonumber \\ &&+
\frac{\partial J^*(\hat{v}_1^*,\hat{v}_0^*,\mathbf{A}_0)}{\partial v_0^*} \frac{\partial \hat{v}_0^*}{\partial v_1^*} \nonumber \\ &=& \mathbf{0}
\end{eqnarray}
and
 $$\frac{\partial J^*_1(\hat{v}_1^*,\mathbf{A}_0)}{\partial \mathbf{A}} =
\frac{\partial J^*(\hat{v}_1^*,\hat{v}_0^*,\mathbf{A}_0)}{\partial \mathbf{A}}= \mathbf{0},$$
so that we may denote
$$\delta J_1^*(\hat{v}_1^*,\mathbf{A}_0)=\mathbf{0}.$$

Furthermore, we may easily  compute,

\begin{eqnarray}
J^*(\hat{v}_1^*,\hat{v}_0^*,\mathbf{A}_0)&=& -F^*(\hat{v}_1^*)+\hat{G}^*(\hat{v}_1^*,\hat{v}_0^*,\mathbf{A}_0)
\nonumber \\ &=& -\langle \phi_0,\hat{v}_1^*\rangle_{L^2}+F(\phi_0,\mathbf{A}_0)+\langle \phi_0,\hat{v}_1^*\rangle_{L^2}-\langle 0,\hat{v}_0^*\rangle_{L^2}-G(\phi_0,\mathbf{A}_0) \nonumber \\ &=& J(\phi_0,\mathbf{A}_0).
\end{eqnarray}

Observe that, in particular, we have
$$J_1^*(\hat{v}_1^*,\mathbf{A}_0)=J^*(\hat{v}_1^*,\hat{v}_0^*,\mathbf{A}_0),$$
where the concerning supremum is attained through the equation
$$\frac{\partial J^*(\hat{v}_1^*, \hat{v}_0^*,\mathbf{A}_0)}{\partial v_0^*}=0,$$
that is
\begin{eqnarray} &&\frac{(\hat{v}_1^*)^2}{(2\hat{v}_0^*-K)^2}-\frac{\hat{v}_0^*}{\alpha}-\beta \nonumber \\ &=&
 |\phi_0|^2-\frac{\hat{v}_0^*}{\alpha}-\beta=0.\end{eqnarray}
Taking the variation in $v_1^*$ in such an equation, we get
$$\frac{2 \hat{v}_1^*}{(2\hat{v}_0^*-K)^2}-\frac{ 4(\hat{v}_1^*)^2}{(2\hat{v}_0^*-K)^3} \frac{\partial \hat{v}^*_0}{\partial v_1^*}-\frac{1}{\alpha}\frac{\partial \hat{v}_0^*}{\partial v_1^*}=0,$$
so that
$$\frac{\partial \hat{v}_0^*}{\partial v_1^*}=\frac{ \frac{2 \phi_0}{(2v_0^*-K)}}{\frac{1}{\alpha}+\frac{4|\phi_0|^2}{2\hat{v}_0^*-K}},$$
where, as previously indicated,
$$\phi_0=\frac{\hat{v}_1^*}{2\hat{v}_0^*-K}.$$

At this point we observe that
\begin{eqnarray}&&\frac{\partial^2 J^*(\hat{v}_1^*,\mathbf{A}_0)}{\partial (v_1^*)^2}  \nonumber \\ &=&
\frac{\partial^2 J^*(\hat{v}_1^*,\hat{v}_0^*,\mathbf{A}_0)}{\partial (v_1^*)^2} \nonumber \\ &&
+\frac{\partial^2 J^*(\hat{v}_1^*,\hat{v}_0^*,\mathbf{A}_0)}{\partial v_1^* \partial v_0^*}\frac{\partial \hat{v}_0^*}{\partial v_1^*} \nonumber \\ &=&
-\frac{1}{\gamma|\nabla-i\rho \mathbf{A}_0|^2+K}-\frac{1}{2\hat{v}_0^*-K} \nonumber \\ &&
+\frac{\frac{4\alpha |\phi_0|^2}{(2\hat{v}_0^*-K)^2}}{\left[1+\frac{4\alpha |\phi_0|^2}{2\hat{v}_0^*-K}\right]}
\nonumber \\ &=&
\frac{-2\hat{v}_0^*-4\alpha |\phi_0|^2+K -\gamma|\nabla -i \rho\mathbf{A}_0|^2-K}{(K+\gamma|\nabla -i \rho\mathbf{A}_0|^2)(2\hat{v}_0^*+4\alpha |\phi_0|^2-K)} \nonumber \\ &=&
\frac{-\delta^2_{\phi\phi}J(\phi_0,\mathbf{A}_0)}{(K+\gamma|\nabla -i \rho\mathbf{A}_0|^2)(2\hat{v}_0^*+4\alpha |\phi_0|^2-K)} \nonumber \\ &>& \mathbf{0}.
\end{eqnarray}

Summarizing,
$$\frac{\partial^2 J^*(\hat{v}_1^*,\mathbf{A}_0)}{\partial (v_1^*)^2} > \mathbf{0}.$$

From these last results,
there exists $r,r_1>0$ such that

\begin{eqnarray}
J(\phi_0,\mathbf{A}_0)&=&\min_{(\phi,\mathbf{A}) \in B_r(\phi_0,\mathbf{A}_0)}J(\phi,\mathbf{A})\nonumber \\ &=&\inf_{v_1^*\in B_{r_1}(\hat{v}_1^*)} J^*_1(v_1^*,\mathbf{A}_0)
\nonumber \\ &=& J^*_1(\hat{v}_1^*,\mathbf{A}_0) \nonumber \\ &=&
\inf_{v_1^* \in B_{r_1}(\hat{v}_1^*)}\left\{ \sup_{v_0^* \in E} J^*(v_1^*,v_0^*,\mathbf{A}_0)\right\}
\nonumber \\ &=& J^*(\hat{v}_1^*,\hat{v}_0^*,\mathbf{A}_0).\end{eqnarray}
The proof is complete.
\end{proof}
\begin{remark} At this point of our  analysis and on, we consider a finite dimensional model version in a finite differences or finite elements context, even though the spaces and operators have not been relabeled. So, also in such a context, the expression $$\int_{\Omega} \frac{(v_1^*)^2}{2v_0^*-K}dx,$$ indeed means
$$(v_1^*)^T(2v_0^*-K\;I_d)^{-1}v_1^*$$ where $I_d$ denotes the identity matrix $n \times n$ and
$$2v_0^*-K\;I_d$$ denotes the diagonal matrix with the vector $$\{2v_0^*(i)-K\}_{n \times 1}$$
as diagonal, for some appropriate $n \in \mathbb{N}$ defined in the discretization process.
\end{remark}

\section{ A second duality principle}

In this section we present another duality principle, which is summarized by the next theorem.

\begin{thm} Let $\Omega \subset \mathbb{R}^3$ be an open, bounded and connected set with a regular
(Lipschitzian) boundary denoted by $\partial \Omega.$

Let $J:U \rightarrow \mathbb{R}$ be a functional defined by

$$J(u)=G(\Lambda u)-F(\Lambda u)-\langle u,f \rangle_{U},\; \forall u \in U,$$
where
$$U=W_0^{1,2}(\Omega),\; f \in L^2(\Omega), Y=Y^*=L^2,$$
$\Lambda :U \rightarrow Y,$ is a bounded linear operator which the
respective adjoint is denoted by $\Lambda^*: Y^* \rightarrow U^*$.

Suppose also $(G \circ \Lambda):U \rightarrow \mathbb{R}$ and $(F \circ \Lambda):U \rightarrow \mathbb{R}$
are Fr\'{e}chet differentiable functionals on $U$ and such that $J$ is bounded below.

Define $G_K^*:Y^* \rightarrow \mathbb{R}$ and $F_K^*:Y^* \rightarrow \mathbb{R}$ by

$$G_K^*(v^*)=\sup_{v \in Y}\left\{\langle v,v^* \rangle_{L^2}-G(v)-\frac{K}{2} \langle v,v \rangle_{L^2}\right\},$$
$$F_K^*(z^*)=\sup_{v \in Y_1}\left\{\langle v,z^* \rangle_{L^2}-F(v)-\frac{K}{2} \langle v,v \rangle_{L^2}\right\}.$$

Assume $(\hat{v}^*,\hat{z}^*, u_0) \in Y^* \times Y^* \times U$ is such that

$$\delta J^*(\hat{v}^*,\hat{z}^*,u_0)= \mathbf{0},$$

where $J^*:Y^* \times Y^* \times U \rightarrow \mathbb{R}$ is defined by

$$J^*(v^*,z^*,u)=-G_K^*(v^*)+F_K^*(z^*)+\langle u,\Lambda^*v^*-\Lambda^*z^*-f \rangle_{U}.$$

Suppose also $\delta^2 J(u_0) > \mathbf{0}$ and $K>0$ is sufficiently big so that

$$G_K^{**}(\Lambda u_0)=G_K(\Lambda u_0),$$ and

$$F_K^{**}(\Lambda u_0)=F_K(\Lambda u_0).$$

Denote also
$$A^+=\{ u \in U\;:\; u u_0  \geq 0, \text{ a.e.  in }\Omega\},$$
$$B^+=\{ u \in U\;:\; \delta^2 J(u)\geq \mathbf{0}\},$$
where we assume, there exists a linear function (in $|u|$) $H$  such that
$$\delta^2J(u)\geq \mathbf{0} \text{ if, and only if, } H(|u|) \geq \mathbf{0}.$$

Moreover, defining the set
$$E=A^+ \cap B^+,$$
we have that $E$ is convex, $$\delta J(u_0)=\mathbf{0},$$ so that
\begin{eqnarray}
J(u_0)&=&\inf_{u \in E} J(u) \nonumber \\ &=&
\inf_{u \in E}\left\{ \inf_{z^* \in Y^*} \left\{\sup_{v^* \in Y^*} J^*(v^*,z^*,u)\right\}\right\}
\nonumber \\ &=&
\inf_{u \in E}\left\{ \sup_{v^* \in Y^*} \left\{\inf_{z^* \in Y^*} J^*(v^*,z^*,u)\right\}\right\}
\nonumber \\ &=& J^*(\hat{v}^*, \hat{z}^*,u_0).
\end{eqnarray}
\end{thm}
\begin{proof}
Observe that
\begin{eqnarray}
J^*(v^*,z^*,u)&=& -G_K^*(v^*)+F^*_K(z^*)+\langle u,\Lambda^*v^*-\Lambda_1^*z^*-f\rangle_{U}
\nonumber \\ &\leq& -\langle \Lambda u,v^* \rangle_{Y}+G(\Lambda u)+\frac{K}{2}\langle \Lambda u, \Lambda u \rangle_{Y} \nonumber \\ && +F_K^*(z^*)+\langle u,\Lambda^*v^*-\Lambda^*z^*-f\rangle_{U}
 \nonumber \\ &=&
 G(\Lambda u)+\frac{K}{2}\langle \Lambda u, \Lambda u \rangle_{Y} -\langle \Lambda u,z^* \rangle_{Y^*}
 +F_K^*(z^*)-\langle u,f \rangle_{L^2},
 \end{eqnarray}
$\; \forall u \in U,\; v^* \in Y^*,\; z^* \in Y^*.$

From this, we obtain
 \begin{eqnarray}
&&\inf_{z^* \in Y^*}\left\{\sup_{v^* \in Y^*}J^*(v^*,z^*,u)\right\}\nonumber \\ &\leq& \inf_{z^* \in Y^*}
\left\{ G(\Lambda u)+\frac{K}{2}\langle \Lambda u, \Lambda u \rangle_{Y} -\langle \Lambda u,z^* \rangle_{Y}
 +F_K^*(z^*)-\langle u,f\rangle_{U}\right\}\nonumber \\ &=& G(\Lambda u)+\frac{K}{2}\langle \Lambda u, \Lambda u \rangle_{Y}
\nonumber \\ &&-F(\Lambda u)-\frac{K}{2}\langle \Lambda u, \Lambda u \rangle_{Y}
\nonumber \\ &&-\langle u,f\rangle_{U} \nonumber \\ &=& J(u).
\end{eqnarray}
Hence, we may infer that
 \begin{eqnarray}\label{us12000}
&&\inf_{u \in E}\left\{\inf_{z^* \in Y_1^*}\left\{\sup_{v^* \in Y^*}J^*(v^*,z^*,u)\right\}\right\}
 \nonumber \\ &\leq& \inf_{u \in E} J(u).
\end{eqnarray}
 On the other hand, from $\delta J^*(\hat{v}^*,\hat{z}^*,u_0)=\mathbf{0},$ we have

$$\frac{\partial G_K^*(\hat{v}^*)}{\partial v^*}-\Lambda u_0=0,$$
so that from the Legendre transform properties, we obtain
$$\hat{v}^*=\frac{\partial G_K(\Lambda u_0)}{\partial v}$$ and
$$G_K^*(\hat{v}^*)=\langle \Lambda u_0, \hat{v}^* \rangle_{Y^*}-G_K(\Lambda
u_0).$$

Similarly, from the variation in $z^*$, we get
$$\frac{\partial F_K^*(\hat{z}^*)}{\partial z^*}-\Lambda u_0=0,$$
so that from the Legendre transform properties, we obtain
$$\hat{z}^*=\frac{\partial F_K(\Lambda u_0)}{\partial v}$$
and
$$F_K^*(\hat{z}^*)=\langle \Lambda u_0, \hat{z}^* \rangle_{Y^*}-F_K(\Lambda
u_0).$$

From the variation in $u$, we have
$$ \Lambda \hat{v}^*-\Lambda^* \hat{z}^*-f=0.$$

Joining the pieces, we have got
\begin{eqnarray}\label{us12005} J^*(\hat{v}^*,\hat{z}^*,u_0)&=&
-G_K^*(\hat{v}^*)+F_K^*(\hat{z}^*) \nonumber \\ &=&
-\langle \Lambda u_0, \hat{v}^* \rangle_{Y}+G_K(\Lambda
u_0)-F_K(\Lambda u_0)+\langle \Lambda u_0, \hat{z}^* \rangle_{Y}
\nonumber \\ &=& G(\Lambda u_0)-F(\Lambda u_0)-\langle u_0,f \rangle_{U} \nonumber \\ &=& J(u_0).
\end{eqnarray}
 Moreover, from
$$ \Lambda \hat{v}^*-\Lambda^* \hat{z}^*-f=0,$$
we also have

$$\Lambda^*\left(\frac{\partial G(\Lambda u_0)}{\partial v}\right)-\Lambda^*\left(\frac{\partial F(\Lambda u_0)}{\partial v}\right)-f=0,$$
so that
$$\delta J(u_0)=\mathbf{0}.$$
Finally, observe that if $u_1,u_2 \in A^+\cap B^+=E$ and $\lambda \in [0,1]$, then $$H(|u_1|)\geq \mathbf{0},$$
$$H(|u_2|)\geq \mathbf{0}$$ and also since $$\text{ sign } u_1=\text{ sign }u_2, \text{ in } \Omega,$$ we get
$$|\lambda u_1+(1-\lambda) u_2|=\lambda |u_1|+(1-\lambda)|u_2|,$$
so that, from the hypotheses on $H$, $$\lambda H(|u_1|)+(1-\lambda)H(|u_2|)=H(|\lambda u_1+(1-\lambda)u_2|)\geq \mathbf{0}$$ and thus, $$\delta^2J(\lambda u_1+(1-\lambda)u_2)\geq \mathbf{0}.$$

From this, we may infer that $E$ is convex.

Moreover, since $J$ is convex in $E$, and $$\delta J(u_0)=\mathbf{0},$$ from
(\ref{us12000}) and (\ref{us12005}), we have that
\begin{eqnarray}
J(u_0)&=&\inf_{u \in E} J(u) \nonumber \\ &=&
\inf_{u \in E}\left\{ \inf_{z^* \in Y^*} \left\{\sup_{v^* \in Y^*} J^*(v^*,z^*,u)\right\}\right\}
\nonumber \\ &=&
\inf_{u \in E}\left\{ \sup_{v^* \in Y^*} \left\{\inf_{z^* \in Y^*} J^*(v^*,z^*,u)\right\}\right\}
\nonumber \\ &=& J^*(\hat{v}^*, \hat{z}^*,u_0).
\end{eqnarray}

The proof is complete.
\end{proof}

\section{ A third duality principle}

Our third duality principle is summarized by the next theorem.

\begin{thm}\label{TH1}
Let $\Omega \subset \mathbb{R}^3$ be an open, bounded and connected set with a regular (Lipschitzian)
boundary denoted by $\partial \Omega.$ Consider the functional $J:U \rightarrow \mathbb{R}$ be defined by
$$J(u)=\frac{\gamma}{2}\int_\Omega \nabla u \cdot \nabla u\;dx+\frac{\alpha}{2}\int_\Omega (u^2-\beta)^2\;dx
-\langle u,f \rangle_{L^2},$$
where $\alpha,\beta,\gamma$ are positive real constants,
$U=W_0^{1,2}(\Omega)$, $f \in L^2(\Omega)$. Here we assume $$-\gamma \nabla^2-2 \alpha \beta < \mathbf{0}$$
in an appropriate matrix sense considering, as above indicated, a finite dimensional not relabeled model approximation, in a finite differences or finite elements context.

Define $F:U \rightarrow \mathbb{R}$
and $G:U \rightarrow \mathbb{R}$, where

$$F(u)=\frac{\gamma}{2}\int_\Omega \nabla u \cdot \nabla u\;dx+\frac{K}{2}\int_\Omega u^2\;dx,$$

and

$$G(u,v)=-\frac{\alpha}{2}\int_\Omega (u^2-\beta+v)^2\;dx+\frac{K}{2} \int_\Omega u^2\;dx+\langle u,f \rangle_{L^2}$$ so that
$$J(u)=F(u)-G(u,0).$$

Define also,

\begin{eqnarray}
F^*(v^*_1)&=&\sup_{u \in U}\{ \langle u,v_1^* \rangle_{L^2}-F(u)\}
\nonumber \\ &=& \frac{1}{2}\int_\Omega \frac{(v_1^*)^2}{K-\gamma\nabla^2}\;dx
\end{eqnarray}

and

\begin{eqnarray}G^*(v_1^*,v_0^*)&=&\sup_{ u \in U}\inf_{v \in L^2} \{ \langle u,v_1^* \rangle_{L^2}-\langle v,v_0^* \rangle_{L^2} -G(u,v)\}
 \nonumber \\ &=&
-\frac{1}{2}\int_\Omega \frac{(v_1^*-f)^2}{2v_0^*-K}\;dx-\frac{1}{2\alpha}\int_\Omega (v_0^*)^2\;dx\nonumber \\
&&-\beta \int_\Omega v_0^*\;dx,\end{eqnarray}
if $-2v_0^*+K>0$ in $\overline{\Omega}.$

Furthermore, denote $$C=\{v_0^* \in C(\overline{\Omega}) \text{ such that } -2v_0^*+K>K/2, \text{ in } \overline{\Omega}\},$$ where $K>0$ is such that $$\frac{1}{\alpha}> \frac{8K_2^2}{K},$$ for some appropriate $K_2>0.$

At this point suppose $u_0 \in U$ is such that $\delta J(u_0)=0,$ $\|\phi_0\|_\infty \leq K_2$ and
$$\delta^2J(u_0)> \mathbf{0}.$$

Define also,
$$A^+=\{u \in U \;:\; u u_0 \geq 0, \text{ in } \Omega\},$$
$$B^+=\{u \in U\;:\; \delta^2J(u)\geq \mathbf{0}\},$$
$$E=A^+ \cap B^+,$$
$$\hat{v}_0^*=\alpha (u_0^2-\beta),$$
$$\hat{v}_1^*=-\gamma \nabla^2 u_0+Ku_0$$.

Under such hypothesis, assuming also $\hat{v}_0^* \in C$ and denoting
$$J^*(v_1,v_0^*)=F^*(v_1^*)-G^*(v_1^*,v_0^*),$$

$$J_1^*(v_1^*)=\sup_{v_0^*\in B} J^*(v_1^*,v_0^*),$$ we have that there exists $r>0$ such that

\begin{eqnarray}
J(u_0)&=& \inf_{u \in E}J(u) \nonumber \\ &=& \inf_{v_1^* \in B_r(\hat{v}_1^*)}J^*(v_1^*) \nonumber \\ &=&
J^*(\hat{v}_1^*) \nonumber \\ &=& \inf_{v_1^* \in B_r(\hat{v}_1^*)}\left\{ \sup_{v_0^* \in C} J^*(v_1^*,v_0^*)\right\}
\nonumber \\ &=& J^*(\hat{v}_1^*,\hat{v}_0^*).
\end{eqnarray}
\end{thm}
\begin{proof} We start by proving that
$$\delta J^*(\hat{v}_1^*,\hat{v}_0^*)=\mathbf{0}.$$

Observe that from $$\frac{\partial J(u_0)}{\partial u}=0,$$ we have
$$ -\gamma\nabla^2 u_0+2\alpha(|u_0|^2-\beta) u_0-f=0, \text{ in } \Omega.$$

 Observe also that \begin{eqnarray}\frac{\partial J^*(\hat{v}_1^*, \hat{v}_0^*)}{\partial v_0^*}
 &=& \frac{(\hat{v}_1^*)^2}{(2\hat{v}_0^*-K)^2}-\frac{\hat{v}_0^*}{\alpha}-\beta \nonumber \\ &=&
 |u_0|^2-\frac{\hat{v}_0^*}{\alpha}-\beta=0.\end{eqnarray}

 Summarizing, we have got
 $$\frac{\partial J^*(\hat{v}_1^*, \hat{v}_0^*)}{\partial v_0^*}=0.$$

Moreover, from the first line in equation (\ref{eq10}), we obtain
$$-\gamma\nabla^2 u_0+Ku_0+2\alpha(|u_0|^2-\beta) u_0-Ku_0-f=0,  \text{ in } \Omega,$$
so that \begin{eqnarray}\hat{v}_1^*&=&-\gamma \nabla^2u_0+K u_0
\nonumber \\ &=&-2\alpha(|u_0|^2-\beta) u_0+K u_0+f.\end{eqnarray}

Hence, $$u_0=-\frac{(\hat{v}_1^*-f)}{2\hat{v}_0^*-K}=\frac{\hat{v}_1^*}{ -\gamma \nabla^2+K},$$
and thus,
 \begin{eqnarray}\frac{\partial J^*(\hat{v}_1^*, \hat{v}_0^*)}{\partial v_1^*}
 &=& -\frac{(\hat{v}_1^*-f)}{2\hat{v}_0^*-K}-\frac{\hat{v}_1^*}{ -\gamma  \nabla^2+K}
 \nonumber \\ &=& u_0-u_0=0.\end{eqnarray}

Such last results may be denoted by
$$\delta J^*(\hat{v}_1^*,\hat{v}_0^*)=\mathbf{0}.$$

Recall now that $$J_1^*(v_1^*)=\sup_{v_0^* \in E} J^*(v_1^*,v_0^*).$$

Thus, \begin{eqnarray} \frac{\partial J^*_1(\hat{v}_1^*)}{\partial v_1^*} &=&
\frac{\partial J^*(\hat{v}_1^*,\hat{v}_0^*)}{\partial v_1^*} \nonumber \\ &&+
\frac{\partial J^*(\hat{v}_1^*,\hat{v}_0^*)}{\partial v_0^*} \frac{\partial \hat{v}_0^*}{\partial v_1^*} \nonumber \\ &=& \mathbf{0}.
\end{eqnarray}
Furthermore, we may easily  compute,
\begin{eqnarray}
J^*(\hat{v}_1^*,\hat{v}_0^*)&=& -F^*(\hat{v}_1^*)+\hat{G}^*(\hat{v}_1^*,\hat{v}_0^*)
\nonumber \\ &=& -\langle u_0,\hat{v}_1^*\rangle_{L^2}+F(\phi_0)+\langle u_0,\hat{v}_1^*\rangle_{L^2}-\langle 0,\hat{v}_0^*\rangle_{L^2}-G(u_0, \mathbf{0}) \nonumber \\ &=& J(u_0).
\end{eqnarray}
Observe that, in particular, we have
$$J_1^*(\hat{v}_1^*)=J^*(\hat{v}_1^*,\hat{v}_0^*),$$
where the concerning supremum is attained through the equation
$$\frac{\partial J^*(\hat{v}_1^*, \hat{v}_0^*)}{\partial v_0^*}=0,$$
that is
\begin{eqnarray} &&\frac{(\hat{v}_1^*-f)^2}{(2\hat{v}_0^*-K)^2}-\frac{\hat{v}_0^*}{\alpha}-\beta \nonumber \\ &=&
 |\phi_0|^2-\frac{\hat{v}_0^*}{\alpha}-\beta=0.\end{eqnarray}
Taking the variation in $v_1^*$ in such an equation, we get
$$\frac{2 (\hat{v}_1^*-f)}{(2\hat{v}_0^*-K)^2}-\frac{ 4(\hat{v}_1^*)^2}{(2\hat{v}_0^*-K)^3} \frac{\partial \hat{v}^*_0}{\partial v_1^*}-\frac{1}{\alpha}\frac{\partial \hat{v}_0^*}{\partial v_1^*}=0,$$
so that
$$\frac{\partial \hat{v}_0^*}{\partial v_1^*}=\frac{ \frac{2 u_0}{(2v_0^*-K)}}{\frac{1}{\alpha}+\frac{4|u_0|^2}{2\hat{v}_0^*-K}},$$
where, as previously indicated,
$$u_0=\frac{\hat{v}_1^*-f}{2\hat{v}_0^*-K}.$$

At this point we observe that
\begin{eqnarray}&&\frac{\partial^2 J^*(\hat{v}_1^*)}{\partial (v_1^*)^2}  \nonumber \\ &=&
\frac{\partial^2 J^*(\hat{v}_1^*,\hat{v}_0^*)}{\partial (v_1^*)^2} \nonumber \\ &&
+\frac{\partial^2 J^*(\hat{v}_1^*,\hat{v}_0^*)}{\partial v_1^* \partial v_0^*}\frac{\partial \hat{v}_0^*}{\partial v_1^*} \nonumber \\ &=&
-\frac{1}{-\gamma \nabla^2+K}-\frac{1}{2\hat{v}_0^*-K} \nonumber \\ &&
+\frac{\frac{4\alpha |u_0|^2}{(2\hat{v}_0^*-K)^2}}{\left[1+\frac{4\alpha |u_0|^2}{2\hat{v}_0^*-K}\right]}
\nonumber \\ &=&
\frac{-2\hat{v}_0^*-4\alpha |u_0|^2+K +\gamma \nabla^2-K}{(K- \gamma \nabla^2)(2\hat{v}_0^*+4\alpha |u_0|^2-K)} \nonumber \\ &=&
\frac{-\delta^2J(u_0)}{(K-\gamma \nabla^2)(2\hat{v}_0^*+4\alpha |u_0|^2-K)} \nonumber \\ &>& \mathbf{0}.
\end{eqnarray}

Summarizing,
$$\frac{\partial^2 J^*(\hat{v}_1^*)}{\partial (v_1^*)^2} > \mathbf{0}.$$
Finally, observe that
$$\delta^2J(u)=-\gamma \nabla^2+6\alpha u^2-2\alpha\beta \geq \mathbf{0},$$ if, and only if
$$H(u)\geq \mathbf{0},$$
where
$$H(u)=\sqrt{6\alpha}|u|-\sqrt{\gamma\nabla^2+2\alpha\beta}\geq \mathbf{0}.$$
Hence, if $u_1,u_2 \in A^+\cap B^+=E$ and $\lambda \in [0,1]$, then $$H(|u_1|)\geq \mathbf{0},$$
$$H(|u_2|)\geq \mathbf{0}$$ and also since $$\text{ sign } u_1=\text{ sign }u_2, \text{ in } \Omega,$$ we get
$$|\lambda u_1+(1-\lambda) u_2|=\lambda |u_1|+(1-\lambda)|u_2|,$$
so that,  $$H(|\lambda u_1+(1-\lambda)u_2|)=H(\lambda |u_1|+(1-\lambda)|u_2|)=\lambda H(|u_1|)+(1-\lambda)H(|u_2|)\geq \mathbf{0}$$ and thus, $$\delta^2J(\lambda u_1+(1-\lambda)u_2)\geq\mathbf{0}.$$

From this, we may infer that $E$ is convex.

From these last results,
there exists $r>0$ such that

\begin{eqnarray}
J(u_0)&=& \inf_{u \in E}J(u) \nonumber \\ &=& \inf_{v_1^* \in B_r(\hat{v}_1^*)}J^*(v_1^*) \nonumber \\ &=&
J^*(\hat{v}_1^*) \nonumber \\ &=& \inf_{v_1^* \in B_r(\hat{v}_1^*)}\left\{ \sup_{v_0^* \in E} J^*(v_1^*,v_0^*)\right\}
\nonumber \\ &=& J^*(\hat{v}_1^*,\hat{v}_0^*).
\end{eqnarray}
The proof is complete.
\end{proof}
\section{ A criterion for global optimality}
In this section we establish a criterion for global optimality.
\begin{thm}\label{TH2}Let $\Omega \subset \mathbb{R}^3$ be an open, bounded and connected set with a regular (Lipschitzian)
boundary denoted by $\partial \Omega.$

Consider the functional $J:U \rightarrow \mathbb{R}$  where

\begin{eqnarray} J(u)&=& \frac{\gamma}{2}\int_\Omega \nabla u \cdot \nabla u \;dx+\frac{\alpha}{2}\int_\Omega
(u^2-\beta)^2\;dx \nonumber \\ && -\langle u,f \rangle_{L^2}\end{eqnarray} where
$\alpha>0,\;\beta>0,\; \gamma>0$, $f \in C^1(\overline{\Omega})$
and $U=W_0^{1,2}(\Omega).$ Suppose also either $$f(x)>0,\; \forall x \in \Omega$$ or $$f(x)<0,\; \forall x \in \Omega.$$

 Here we assume $$-\gamma \nabla^2-2 \alpha \beta < \mathbf{0}$$
in an appropriate matrix sense considering, as above indicated, a finite dimensional not relabeled model approximation, in a finite differences or finite elements context.

Define
$$A^+=\{u \in U \;:\; u f \geq 0, \text{ in } \overline{\Omega}\},$$
$$B^+=\{u \in U\;:\; \delta^2J(u)\geq \mathbf{0}\}$$ and
$$E=A^+ \cap B^+.$$

Under such hypotheses, $E$ is convex and

$$\inf_{u \in E}J(u)=\inf_{u \in U} J(u)$$

\end{thm}
\begin{proof} Define $$\eta=\inf_{u \in U} J(u).$$

Let $\varepsilon>0$

Hence, by density there exists $u_\varepsilon \in C^1_c(\Omega)$ such that
$$\delta^2J(u_\varepsilon) \geq  \mathbf{0}$$ and
$$\eta \leq J(u_\varepsilon) < \eta +\varepsilon.$$

Define  \begin{equation}
v_\varepsilon(x)=\left \{
\begin{array}{ll}
 u_\varepsilon(x), &  \text{ if }\; u_\varepsilon(x)f(x) \geq 0,
 \\
 -u_\varepsilon(x), &  \text{ if }\; u_\varepsilon(x)f(x) < 0,
  \end{array} \right.\end{equation}
$\forall x \in \overline{\Omega}.$

Observe that $$\delta^2J(v_\varepsilon)=\delta^2 J(u_\varepsilon) \geq \mathbf{0}$$
and
\begin{eqnarray} J(v_\varepsilon)&=& \frac{\gamma}{2}\int_\Omega \nabla v_\varepsilon \cdot \nabla v_\varepsilon \;dx+\frac{\alpha}{2}\int_\Omega
(v_\varepsilon^2-\beta)^2\;dx \nonumber \\ && -\langle v_\varepsilon,f \rangle_{L^2}
\nonumber \\ &\leq& \frac{\gamma}{2}\int_\Omega \nabla u_\varepsilon \cdot \nabla u_\varepsilon \;dx+\frac{\alpha}{2}\int_\Omega
(u_\varepsilon^2-\beta)^2\;dx \nonumber \\ && -\langle u_\varepsilon,f \rangle_{L^2} \nonumber \\ &=&
J(u_\varepsilon).\end{eqnarray}

Hence $$\eta \leq J(v_\varepsilon) \leq J(u_\varepsilon) < \eta+\varepsilon.$$

From this, since $v_\varepsilon \in E$,  we obtain
$$\eta \leq \inf_{u \in E}J(u) < \eta+\varepsilon.$$

Since $\varepsilon>0$ is arbitrary, we may infer that
$$\inf_{u \in U}J(u)=\eta=\inf_{ u \in E}J(u).$$

Finally, observe also that
$$\delta^2J(u)=-\gamma \nabla^2+6\alpha u^2-2\alpha\beta \geq \mathbf{0},$$ if, and only if
$$H(u)\geq \mathbf{0},$$
where
$$H(u)=\sqrt{6\alpha}|u|-\sqrt{\gamma\nabla^2+2\alpha\beta}\geq\mathbf{0}.$$
Hence, if $u_1,u_2 \in A^+\cap B^+=E$ and $\lambda \in [0,1]$, then $$H(|u_1|)\geq \mathbf{0},$$
$$H(|u_2|)\geq \mathbf{0}$$ and also since $$\text{ sign } u_1=\text{ sign }u_2, \text{ in } \Omega,$$ we get
$$|\lambda u_1+(1-\lambda) u_2|=\lambda |u_1|+(1-\lambda)|u_2|,$$
so that,  $$H(|\lambda u_1+(1-\lambda)u_2|)=H(\lambda |u_1|+(1-\lambda)|u_2|)=\lambda H(|u_1|)+(1-\lambda)H(|u_2|)\geq \mathbf{0}$$ and thus, $$\delta^2J(\lambda u_1+(1-\lambda)u_2)\geq\mathbf{0}.$$

From this, we may infer that $E$ is convex.

The proof is complete.
\end{proof}

\subsection{ The concerning duality principle}

In this section we develop a duality principle concerning the last optimality criterion established.
\begin{thm}
Let $\Omega \subset \mathbb{R}^3$ be an open, bounded and connected set with a regular (Lipschitzian)
boundary denoted by $\partial \Omega.$ Consider the functional $J:U \rightarrow \mathbb{R}$ be defined by
$$J(u)=\frac{\gamma}{2}\int_\Omega \nabla u \cdot \nabla u\;dx+\frac{\alpha}{2}\int_\Omega (u^2-\beta)^2\;dx
-\langle u,f \rangle_{L^2},$$
where $\alpha,\beta,\gamma$ are positive real constants,
$U=W_0^{1,2}(\Omega)$, $f \in C^1(\overline{\Omega})$ and we also denote $Y=Y^*=L^2(\Omega).$

Here we assume $$-\gamma \nabla^2-2 \alpha \beta < \mathbf{0}$$
in an appropriate matrix sense considering, as above indicated, a finite dimensional not relabeled model approximation, in a finite differences or finite elements context.

Assume also either $$f(x)>0,\; \forall x \in \Omega$$ or $$f(x)<0,\;\forall x \in \Omega.$$ 

Given $u \in U$ define
$$L_1(u)=\sup_{v_0^* \in Y^*} \left\{ \int_\Omega v_0 u^2\;dx-\frac{1}{2\alpha}\int_\Omega (v_0^*)^2
-\beta \int_\Omega v_0^*\;dx \right\},$$
and
$$L_2(u)=\sup_{v_0^* \in B^*} \left\{ \int_\Omega v_0 u^2\;dx-\frac{1}{2\alpha}\int_\Omega (v_0^*)^2
-\beta \int_\Omega v_0^*\;dx \right\},$$
where $$B^*=\{v_0^* \in Y^*\;:\; -2v_0^*+K > K/2 \text{ in } \overline{\Omega}\}$$ for some $K>0$ to be specified.

Let
$$U_1=\{u \in U \text{ such that } L_1(u)=L_2(u)\text{ and } \|u\|_{1,\infty} \leq \sqrt[4]{K}\}.$$

Moreover, define  $F:U \rightarrow \mathbb{R}$
and $G:U \rightarrow \mathbb{R}$, where

$$F(u)=\frac{\gamma}{2}\int_\Omega \nabla u \cdot \nabla u\;dx+\frac{K}{2}\int_\Omega u^2\;dx,$$

and

$$G(u,v)=-\frac{\alpha}{2}\int_\Omega (u^2-\beta+v)^2\;dx+\frac{K}{2} \int_\Omega u^2\;dx+\langle u,f \rangle_{L^2}$$ so that
$$J(u)=F(u)-G(u,0).$$

Define also,

$$B_1^*=\{v_1^* \in Y^*\;:\; \|v_1^*-f\|_\infty \leq K_2\},$$
where $K_2>0$ and $K>0$ are such that
$$-\frac{32 K_2^2}{K^3}+\frac{1}{\alpha}>0,$$

\begin{eqnarray}C_1&=&\left\{v_1^* \in B_1^*\;:\; \text{ there exists } u \in U_1 \right. \nonumber \\ && \left.
\text{ such that } v_1^*=\frac{\partial F(u)}{\partial u} \right\},\end{eqnarray}

\begin{eqnarray}C_2&=&\left\{v_1^* \in B_1^*\;:\; \text{ there exists } u \in U_1 \right. \nonumber \\ && \left.
\text{ such that } v_1^*=\frac{\partial G(u,0)}{\partial u} \right\},\end{eqnarray}
and
$$C^*=C_1 \cap C_2.$$

Furthermore, define $F^*:C^* \rightarrow \mathbb{R}$ by
\begin{eqnarray}
F^*(v^*_1)&=&\sup_{u \in U_1}\{ \langle u,v_1^* \rangle_{L^2}-F(u)\}
\nonumber \\ &=& \frac{1}{2}\int_\Omega \frac{(v_1^*)^2}{K-\gamma\nabla^2}\;dx
\end{eqnarray}

and $G^*:C^* \times B^* \rightarrow \mathbb{R}$ by

\begin{eqnarray}G^*(v_1^*,v_0^*)&=&\sup_{ u \in U}\inf_{v \in L^2} \{ \langle u,v_1^* \rangle_{L^2}-\langle v,v_0^* \rangle_{L^2} -G(u,v)\}
 \nonumber \\ &=&
-\frac{1}{2}\int_\Omega \frac{(v_1^*-f)^2}{2v_0^*-K}\;dx-\frac{1}{2\alpha}\int_\Omega (v_0^*)^2\;dx\nonumber \\
&&-\beta \int_\Omega v_0^*\;dx.\end{eqnarray}

Define also,
$$A^+=\{u \in U \;:\; u f \geq 0, \text{ in } \overline{\Omega}\},$$
$$B^+=\{u \in U\;:\; \delta^2J(u)\geq\mathbf{0}\},$$
$$E=A^+ \cap B^+,$$ and $$E_1=E \cap U_1.$$
Moreover, define
$$\hat{v}_0^*=\alpha (u_0^2-\beta),$$
$$\hat{v}_1^*=(2v_0^*-K)u_0+f$$
and assume $u_0 \in U$ is such that $\delta J(u_0)=\mathbf{0},$ and
 $$u_0 \in E_1,$$
Under such hypothesis, assuming also $\hat{v}_0^* \in B^*$, $\hat{v}_1^* \in C^*$ and denoting
$$J^*(v_1,v_0^*)=-F^*(v_1^*)+G^*(v_1^*,v_0^*),$$ we have

\begin{eqnarray}
J(u_0)&=& \inf_{u \in E_1}J(u) \nonumber \\ &=& \inf_{u \in U}J(u)
\nonumber \\ &=&  \inf_{v_1^* \in C^*}\left\{ \sup_{v_0^* \in B^*} J^*(v_1^*,v_0^*)\right\}
\nonumber \\ &=& J^*(\hat{v}_1^*,\hat{v}_0^*).
\end{eqnarray}
\end{thm}
\begin{proof} Define $$\eta=\inf_{u \in U} J(u).$$

Hence
\begin{eqnarray}
\eta &\leq& J(u) \nonumber \\ &\leq&
-\langle u,v_1^*\rangle_{L^2}+F(u) \nonumber \\ && +\sup_{u \in U_1}\{\langle u,v_1^*\rangle_{L^2}-G(u,0)\}
\nonumber \\ &=& -\langle u,v_1^*\rangle_{L^2}+F(u) \nonumber \\ && +\sup_{u \in U_1}\left\{\sup_{v_0^* \in B^*}\left\{\langle u,v_1^*\rangle_{L^2}+\int_\Omega v_0^*u^2\;dx-\frac{K}{2}\int_\Omega u^2\;dx-\frac{1}{2\alpha}\int_\Omega (v_0^*)^2\;dx-\beta \int_\Omega v_0^*\;dx\right\}\right\},\nonumber\end{eqnarray}
$\forall u \in U_1, v_1^* \in C^*.$

Thus,
\begin{eqnarray}
\eta &\leq& J(u) \nonumber \\ &\leq&
-\langle u,v_1^*\rangle_{L^2}+F(u) \nonumber \\ && +\sup_{u \in U_1}\{\langle u,v_1^*\rangle_{L^2}-G(u,0)\}
\nonumber \\ &=& -\langle u,v_1^*\rangle_{L^2}+F(u) \nonumber \\ && +\sup_{v_0^* \in B^*}\left\{\sup_{u \in U_1}\left\{\langle u,v_1^*\rangle_{L^2}+\int_\Omega v_0^*u^2\;dx-\frac{K}{2}\int_\Omega u^2\;dx-\frac{1}{2\alpha}\int_\Omega (v_0^*)^2\;dx -\beta \int_\Omega v_0^*\;dx\right\}\right\} \nonumber \\ &=& -\langle u,v_1^*\rangle_{L^2}+F(u) \nonumber \\ &&
+\sup_{v_0^* \in B^*} G^*(v_1^*,v_0^*),\nonumber\end{eqnarray}
$\forall u \in U_1, v_1^* \in C^*.$

From this, we obtain
\begin{eqnarray}
\eta &\leq& \inf_{u \in U_1}\{ -\langle u,v_1^*\rangle_{L^2}+F(u)\} \nonumber \\ &&
+\sup_{v_0^* \in B^*} G^*(v_1^*,v_0^*) \nonumber \\ &=&-F^*(v_1^*)+\sup_{v_0^* \in B^*} G^*(v_1^*,v_0^*) \nonumber
\\ &=& \sup_{v_0^* \in B^*}J^*(v_1^*,v_0^*),\end{eqnarray}
$\forall v_1^* \in C^*.$

Summarizing, we have got
\begin{equation}\label{uk1001}\inf_{u \in U}J(u) \leq \inf_{v_1^* \in C^*}\left\{ \sup_{v_0^* \in B^*} J^*(v_1^*,v_0^*)\right\}. \end{equation}

Similarly as in the proof of the Theorem \ref{TH1}, we may obtain
$$\delta J^*(\hat{v}_1^*,\hat{v}_0^*)=\mathbf{0},$$
$$J^*(\hat{v}_1^*,\hat{v}_0^*)=J(u_0),$$
and
$$J^*(\hat{v}_1^*,\hat{v}_0^*)=\sup_{v_0^* \in B^*} J^*(\hat{v}_1^*,v_0^*).$$

From the proof of Theorem \ref{TH2} we may infer that $E$ is convex.

From this, since $u_0 \in E_1 \subset E$ and $\delta J(u_0)=\mathbf{0}$ and $E$ is convex we have that, also from the Theorem \ref{TH2},
$$J(u_0)=\inf_{u \in E}J(u)=\inf_{u \in U}J(u).$$

Consequently, from such a result, from $\hat{v}_1^* \in C^*$ and $(\ref{uk1001})$ we have that
\begin{eqnarray}
J(u_0)&=& \inf_{u \in E_1}J(u) \nonumber \\ &=& \inf_{u \in U}J(u)
\nonumber \\ &=&  \inf_{v_1^* \in C^*}\left\{ \sup_{v_0^* \in B^*} J^*(v_1^*,v_0^*)\right\}
\nonumber \\ &=& J^*(\hat{v}_1^*,\hat{v}_0^*).
\end{eqnarray}
The proof is complete.
\end{proof}
\section{Numerical results}
In this section we present some numerical results for the following sets
$$\Omega=[-1/2,1/2]\times [-1/2,1/2]\times [-1/2,1/2],$$
and
$$\Omega_1=[-3/2,3/2]\times [-3/2,3/2]\times [-3/2,3/2].$$

The system of equation in question, namely, the complex Ginzburg-Landau one, is given by

\begin{eqnarray}
&&-\gamma \nabla^2\phi-2[i\rho\gamma ((\mathbf{A}\cdot \nabla \phi)+\text{ div } \mathbf{A} \phi)]
\nonumber \\ && +\gamma\rho^2|\mathbf{A}|^2\phi+\alpha |\phi|^2 \phi-\beta \phi=0, \text{ in } \Omega,
\end{eqnarray}

$$(\nabla \phi-i\rho \mathbf{A}\phi)\cdot \mathbf{n}=0, \text{ on } \partial \Omega,$$

$$K_0\text{ curl curl }\mathbf{A}=K_0\text{ curl }\mathbf{B}_0+\tilde{J}, \text{ in }\Omega,$$

$$\text{ curl curl }\mathbf{A}=\text{ curl } \mathbf{B}_0, \text{ in } \Omega_1\setminus \overline{\Omega}.$$

Here
$$\tilde{J}=-2Re[ i \rho\gamma \phi^*\nabla \phi]-\rho^2\gamma |\phi|^2\mathbf{A}.$$

At this point we start to describe the process concerning the numerical method of lines.

Fixing a starting point $\{\hat{\phi}_0\}$ and  $\{(\mathbf{A}_0)_n\}$ and considering the generalized method of lines, we discretize the system in partial finite differences in $z$, that is, we obtain $N-1$ lines which correspond to the $N-1$ partial  differential equations in $(x,y).$

\begin{eqnarray}
&&-\gamma \frac{(\phi_{n+1}-2\phi_n+\phi_{n-1})}{d^2}-\gamma \nabla^2\phi_n+K\phi_n-K(\hat{\phi}_0)_n
\nonumber \\ &&-2[i\rho\gamma ((\mathbf{A}_0)_n \nabla (\hat{\phi}_0)_n+div(\mathbf{A}_0)(\hat{\phi}_0)_n)]
\nonumber \\ && +\rho^2\gamma |(\mathbf{A}_0)_n|^2 \phi_n+\alpha |(\hat{\phi}_0)_n|^2\phi_n-\beta \phi_n=0
\end{eqnarray}
$\forall n \in \{1,\ldots,N-1\},$ where $d=1/N.$

With such a equation in mind, we denote

\begin{equation}\label{uk400.a}\phi_{n+1}-2\phi_n+\phi_{n-1}-K\phi_n \frac{d}{\gamma}+T_n(\phi_n)\frac{d^2}{\gamma}=0,
\end{equation}
where
\begin{eqnarray}T_n(\phi_n)&=&-\gamma \nabla^2\phi_n+2[i\rho\gamma((\mathbf{A}_0)_n \nabla (\hat{\phi}_0)_n
+\text{ div } (\mathbf{A}_0)_n (\phi_0)_n)]
\nonumber \\ &&-K(\hat{\phi}_0)_n-\rho^2\gamma |(\mathbf{A}_0)_n|^2\phi_n-\alpha |(\hat{\phi}_0)_n|^2\phi_n+\beta \phi_n.
\end{eqnarray}

For $n=1$, from the boundary condition
$$ (\nabla \phi-i\rho\mathbf{A}\phi)\cdot \mathbf{n}=0$$ at
$x=-1/2$ we get
$$\phi_0=H_1 \phi_1,$$ fora n appropriate matrix $H_1.$

Replacing such a relation in (\ref{uk400.a}), we obtain
$$\phi_2-2\phi_1+H_1\phi_1-K\phi_1 \frac{d^2}{\gamma}+T_1(\phi_1)\frac{d^2}{\gamma}=0,$$
so that
\begin{equation}\label{us400.a}\phi_1=a_1\phi_2+b_1T_1(\phi_1)\frac{d^2}{\gamma}+E_1,\end{equation}
where
$$a_1=\left(2+K\frac{d^2}{\gamma}-H_1\right)^{-1},$$
$$b_1=a_1,$$
$$E_1=0.$$
For $n=2$ replacing (\ref{us400.a}) into (\ref{uk400.a}), we obtain
\begin{eqnarray}
&&\phi_3-2\phi_2+a_1\phi_2+b_1T_1(\phi_1)\frac{d^2}{\gamma}
\nonumber \\ &&-K\phi_2 \frac{d^2}{\gamma}+T_2(\phi_2)\frac{d^2}{\gamma}=0.\end{eqnarray}

From this, we may write
$$\phi_2=a_2\phi_3+b_2T(\phi_2)\frac{d^2}{\gamma}+E_2,$$
where
$$a_2=\left(2-a_1+K\frac{d^2}{\gamma}\right)^{-1},$$
$$b_2=a_2(b_1+1)$$
and
$$E_2=a_2b_1(T_1(\phi_1)-T_2(\phi_2))\frac{d^2}{\gamma}.$$

At this point we remark that the matrix concerning the operator $\nabla^2$ must take into account the boundary conditions in $(y,z)$ at each $n \in \{1,\ldots,N-1\}.$

reasoning inductively, having
$$\phi_{n-1}=a_{n-1}\phi_n+b_{n-1}T_{n-1}(\phi_{n-1})\frac{d^2}{\gamma}+E_{n-1}$$ and replacing such an relation into
(\ref{uk400.a}) we obtain
\begin{eqnarray}
&&\phi_{n+1}-2\phi_n+a_{n-1}\phi_n+b_{n-1}T_{n-1}(\phi_{n-1})\frac{d^2}{\gamma}+E_{n-1}
\nonumber \\ &&-K\phi_n \frac{d^2}{\gamma}+T_n(\phi_n)\frac{d^2}{\gamma}=0.\end{eqnarray}
Hence,
$$\phi_{n}=a_{n}\phi_n+b_{n}T_n(\phi_{n})\frac{d^2}{\gamma}+E_{n}$$
where
$$a_n=\left(2-a_{n-1}+K \frac{d^2}{\gamma}\right)^{-1},$$
$$b_n=a_n(b_{n-1}+1),$$
$$E_n=a_nb_{n-1}(T_{n-1}(\phi_{n-1})-T_n(\phi_n))\frac{d^2}{\gamma}+a_nE_{n-1}.$$

Thus, for $n=N-1$ for the boundary condition
$$(\nabla \phi-i\rho\mathbf{A}\phi)\cdot \mathbf{n}=0$$ at $x=1/2$, we get
$$\phi_N=H_2\phi_{N-1}$$ for an appropriate matrix $H_2$.

Hence, from the previous results, with $n=N-1$, we get
\begin{eqnarray}
\phi_{N-1}&=& a_{N-1}\phi_N+b_{N-1}T_{N-1}(\phi_{N-1})\frac{d^2}{\gamma}+E_{N-1} \nonumber \\
&\approx& a_{N-1}H_2\phi_{N-1}+b_{N-1}T_{N-1}(\phi_{N-1})\frac{d^2}{\gamma}.
\end{eqnarray}

Solving this last linear partial differential equation we obtain $\phi_{N-1}$.

Having $\phi_{N-1}$ we obtain $\phi_{N-2}$ through the equation
$$\phi_{N-2}\approx a_{N-2}\phi_{N-1}+b_{N-2}T_{N-2}(\phi_{N-2})\frac{d^2}{\gamma}.$$

Having $\phi_{N-2}$ similarly we obtain $\phi_{N-3}$ and so on up to finding $\phi_1.$

Having $\{\phi_n\}$ the next step is to calculate $\mathbf{A}=\{\mathbf{A}_n\}$ through the linear equations
$$K_0 \text{ curl curl }\mathbf{A}=K_0\text{ curl }\mathbf{B}_0+\tilde{J}(\{\phi_n\},\mathbf{A}), \text{ in }\Omega,$$

$$\text{ curl curl }\mathbf{A}=\text{ curl } \mathbf{B}_0, \text{ in } \Omega_1\setminus \overline{\Omega}.$$

The idea here is to fix the Gauge of London through the equation
$$\text{ div }\mathbf{A}=0 \text{ in } \Omega_1,$$
with the boundary conditions
$$\mathbf{A}\cdot \mathbf{n}=0, \text{ on } \partial \Omega_1.$$

Finally, we replace $\hat{\phi}_0$ and $\mathbf{A}_0$ by $\{\phi_n\}$ and $\{\mathbf{A}_n\}$ and repeat the process
until an appropriate convergence criterion is satisfied.

\subsection{A numerical example}

We present numerical results for $\gamma=\alpha=\beta=K(0)=1$. In this example
$$\mathbf{B}_0(x,y,y)=B_0( f(x,y)\mathbf{i}+f(x,y)\mathbf{j})$$
where $$f(x,y)=(-3/2+x)^2(-3/2+y)^2(-3/2+z)^2(x+3/2)(y+3/2)(z+3/2)/3^6$$
and $$B_0=0.008.$$

For the solution $|\phi(x,y,0)|^2$ at the section
$z=0$, please see \ref{LGNOVEMBER-2020-5}.

For the solutions for $A_1(x,y,0)$ and $A_2(x,y,0)$ 
please see figures \ref{LGNOVEMBER-2020-6} and \ref{LGNOVEMBER-2020-7}. 
\begin{figure}
\centering \includegraphics [width=3in]{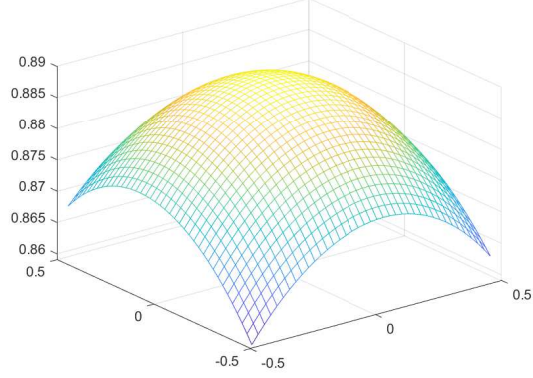}
\\ \caption{\small{ Solution  $|\phi|^2(x,y,0)$ for the section z=0 for $B_0=0.008$} }\label{LGNOVEMBER-2020-5}
\end{figure}
\begin{figure}
\centering \includegraphics [width=3in]{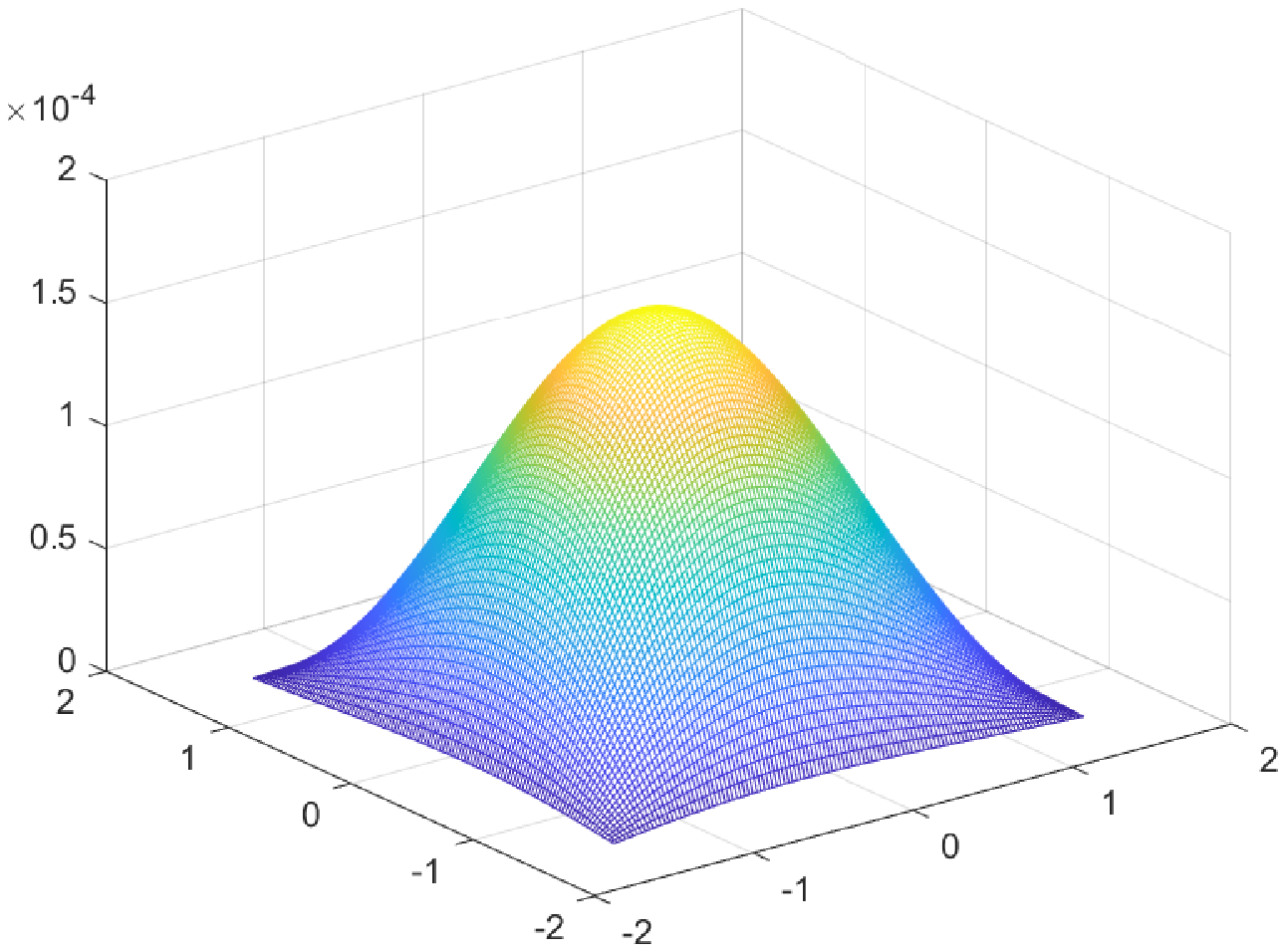}
\\ \caption{\small{ Solution  $A_1(x,y,0)$ for the section z=0 for $B_0=0.008$} }\label{LGNOVEMBER-2020-6}
\end{figure}
\begin{figure}
\centering \includegraphics [width=3in]{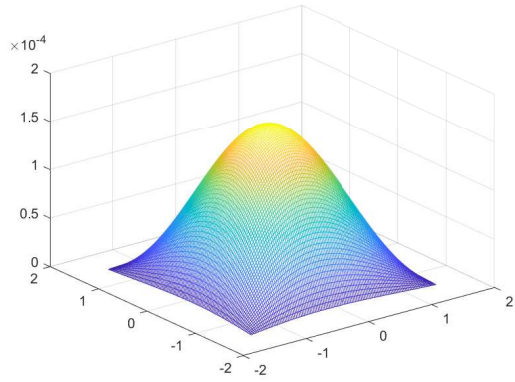}
\\ \caption{\small{ Solution  $A_2(x,y,0)$ for the section z=0 for $B_0=0.008$} }\label{LGNOVEMBER-2020-7}
\end{figure}


For $B_0=0.031$, for the solution $|\phi(x,y,0)|^2$ at the section
$z=0$, please see \ref{LGNOVEMBER-2020-1}.

For such a $B_0$ value, for the solutions for $A_1(x,y,0)$ and $A_2(x,y,0)$ 
please see figures \ref{LGNOVEMBER-2020-2} and \ref{LGNOVEMBER-2020-3}. 

\begin{figure}
\centering \includegraphics [width=3in]{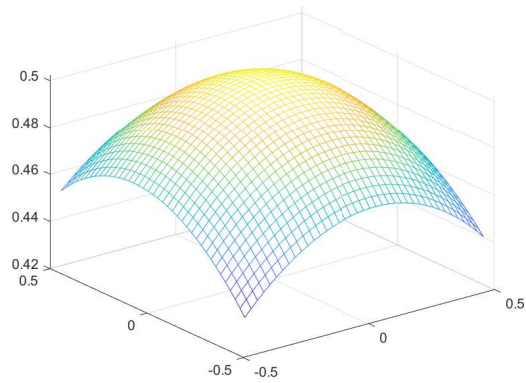}
\\ \caption{\small{ Solution  $|\phi|^2(x,y,0)$ for the section z=0 for $B_0=0.031$} }\label{LGNOVEMBER-2020-1}
\end{figure}
\begin{figure}
\centering \includegraphics [width=3in]{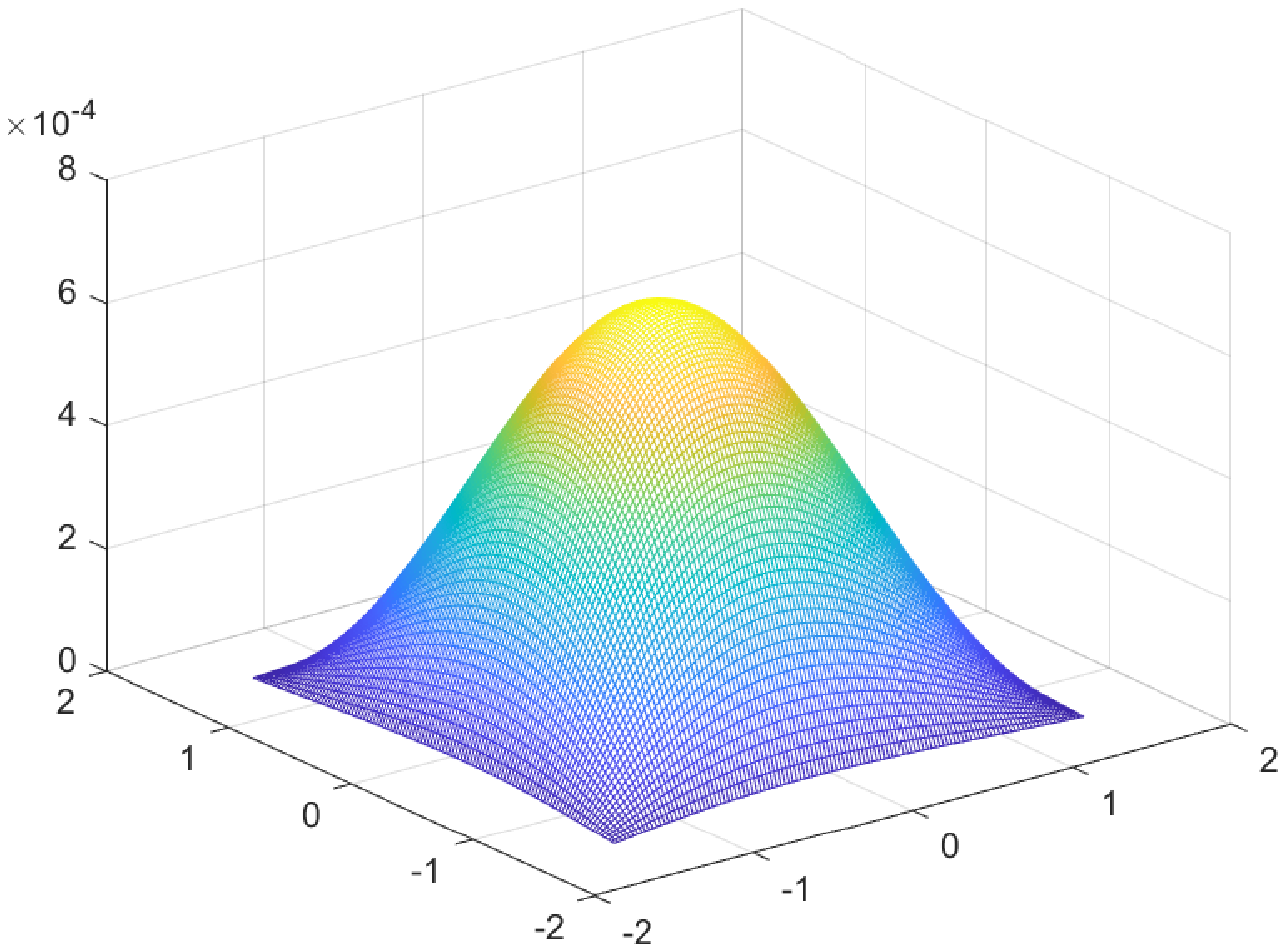}
\\ \caption{\small{ Solution  $A_1(x,y,0)$ for the section z=0 for $B_0=0.031$} }\label{LGNOVEMBER-2020-2}
\end{figure}
\begin{figure}
\centering \includegraphics [width=3in]{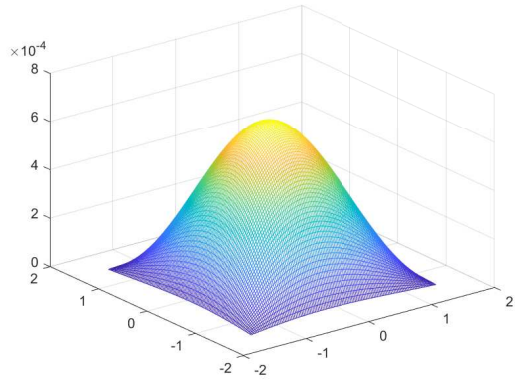}
\\ \caption{\small{ Solution  $A_2(x,y,0)$ for the section z=0 for $B_0=0.031$} }\label{LGNOVEMBER-2020-3}
\end{figure}
\begin{remark}
We observe that for both values of $B_0$ the effect of magnetic field on the $|\phi|^2$ distribution is more present close to the
boundaries of $\Omega$. Also as expected, the higher value of $B_0$ corresponds to more decreasing in the $|\phi|^2$ distribution on its domain. We recall that $|\phi|^2$ is point-wise the proportion of electrons along the sample  in the super-conducting state. It is  always expected for $|\phi|^2$ point-wise,  a value between $0$ and $1$, with $\phi=0$ corresponding to the normal state and $\phi=1$ corresponding to the super-conducting state.
\end{remark}

\section{Initial model formulation} In the present section, in a first step,  we develop  a new existence proof and  a dual variational formulation for the Kirchhoff-Love thin plate model. Previous results on existence in mathematical elasticity and related models may be found in \cite{903,[3],[4]}.

At this point we refer to the exceptionally important article "A contribution to contact problems for a class of solids and structures" by
W.R. Bielski and J.J. Telega, \cite{85},  published in 1985,  as the first one to successfully  apply and generalize the convex analysis approach to a model in non-convex and non-linear mechanics.

The present work is, in some sense, a kind of extension  of this previous work \cite{85} and others such as \cite{2900}, which greatly influenced and
inspired my work and recent book \cite{120}.

Here we highlight that such earlier results establish the complementary energy under the hypothesis of positive definiteness of the membrane force tensor at a critical point (please see  \cite{85,2900} for details).

We have obtained a dual variational formulation which allows the global optimal point in question not to be positive definite (for  related results see F.Botelho \cite{120}), but also not necessarily negative definite. The approach developed also includes sufficient conditions of optimality for the primal problem.
It is worth mentioning that the standard tools of convex analysis used in this text may be found in \cite{[6],120}, for example.

At this point we start to describe the primal formulation.

    Let $\Omega\subset\mathbb{R}^2$ be an open, bounded, connected set  which  represents the middle surface of a plate
of thickness $h$. The boundary of $\Omega$, which is assumed to be regular (Lipschitzian), is
denoted by $\partial \Omega$. The vectorial basis related to the cartesian
system $\{x_1,x_2,x_3\}$ is denoted by $( \textbf{a}_\alpha,
\textbf{a}_3)$, where $\alpha =1,2$ (in general Greek indices stand
for 1 or 2), and where $\textbf{a}_3$ is the vector normal to $\Omega$, whereas $\textbf{a}_1$ and $\textbf{a}_2$ are orthogonal vectors parallel to $\Omega.$  Also,
$\textbf{n}$ is
the outward normal to the plate surface.

    The displacements will be denoted by
$$ \hat{\textbf{u}}=\{\hat{u}_\alpha,\hat{u}_3\}=\hat{u}_\alpha
\textbf{a}_\alpha+ \hat{u}_3 \textbf{a}_3.
$$

The Kirchhoff-Love relations are
\begin{eqnarray}&&\hat{u}_\alpha(x_1,x_2,x_3)=u_\alpha(x_1,x_2)-x_3
w(x_1,x_2)_{,\alpha}\; \nonumber \\ && \text{ and }
\hat{u}_3(x_1,x_2,x_3)=w(x_1,x_2).
\end{eqnarray}

Here $ -h/2\leq x_3 \leq h/2$ so that we have
$u=(u_\alpha,w)\in U$ where
\begin{eqnarray}
U&=&\left\{(u_\alpha,w)\in W^{1,2}(\Omega; \mathbb{R}^2) \times
W^{2,2}(\Omega) , \right.\nonumber \\ &&\; u_\alpha=w=\frac{\partial w}{\partial \textbf{n}}=0 \left.
\text{ on } \partial \Omega \right\} \nonumber \\ &=&W_0^{1,2}(\Omega; \mathbb{R}^2) \times W_0^{2,2}(\Omega).\nonumber\end{eqnarray}
It is worth emphasizing that the boundary conditions here specified refer to a clamped plate.

We define the operator $\Lambda: U \rightarrow Y \times Y$, where $Y=Y^*=L^2(\Omega; \mathbb{R}^{2 \times 2})$, by
$$\Lambda(u)=\{\gamma(u), \kappa(u)\},$$
$$\gamma_{\alpha\beta}(u)= \frac{u_{\alpha,\beta}+u_{\beta,\alpha}}{2}+\frac{w_{,\alpha} w_{,\beta}}{2},$$
$$\kappa_{\alpha \beta}(u)=-w_{,\alpha \beta}.$$
The constitutive relations are given by
\begin{equation}
N_{\alpha\beta}(u)=H_{\alpha\beta\lambda\mu} \gamma_{\lambda\mu}(u),
\end{equation}
\begin{equation}
M_{\alpha\beta}(u)=h_{\alpha\beta\lambda\mu} \kappa_{\lambda\mu}(u),
\end{equation}
where: $\{H_{\alpha\beta\lambda\mu}\}$
 and
 $\{h_{\alpha\beta\lambda\mu}=\frac{h^2}{12}H_{\alpha\beta\lambda\mu}\}$,
 are symmetric positive definite   fourth order tensors. From now on, we denote $\{\overline{H}_{\alpha\beta \lambda \mu}\}=\{H_{\alpha\beta \lambda \mu}\}^{-1}$ and $\{\overline{h}_{\alpha\beta \lambda \mu}\}=\{h_{\alpha\beta \lambda \mu}\}^{-1}$.

 Furthermore
 $\{N_{\alpha\beta}\}$ denote the membrane force tensor and
 $\{M_{\alpha\beta}\}$ the moment one.
    The plate stored energy, represented by $(G\circ
    \Lambda):U\rightarrow\mathbb{R}$ is expressed by
 \begin{equation}\label{80} (G\circ\Lambda)(u)=\frac{1}{2}\int_{\Omega}
  N_{\alpha\beta}(u)\gamma_{\alpha\beta}(u)\;dx+\frac{1}{2}\int_{\Omega}
  M_{\alpha\beta}(u)\kappa_{\alpha\beta}(u)\;dx
  \end{equation}
 and the external work, represented by $F:U\rightarrow\mathbb{R}$, is given by
   \begin{equation}\label{81} F(u)=\langle w,P \rangle_{L^2(\Omega)}+\langle u_\alpha,P_\alpha \rangle_{L^2(\Omega)}
,\end{equation} where $P, P_1, P_2 \in L^2(\Omega)$ are external loads in the directions $\textbf{a}_3$, $\textbf{a}_1$ and $\textbf{a}_2$ respectively. The potential energy, denoted by
$J:U\rightarrow\mathbb{R}$ is expressed by:
$$
J(u)=(G\circ\Lambda)(u)-F(u)
$$

Finally, we also emphasize from now on, as their meaning are clear, we may denote $L^2(\Omega)$ and $L^2(\Omega; \mathbb{R}^{2 \times 2})$ simply by $L^2$, and the respective norms by $\|\cdot \|_2.$ Moreover derivatives are always understood in the distributional sense, $\mathbf{0}$ may denote the zero vector in  appropriate Banach spaces and, the following and relating notations are used:
$$w_{,\alpha\beta}=\frac{\partial^2 w}{\partial x_\alpha \partial x_\beta},$$
$$u_{\alpha,\beta}=\frac{\partial u_\alpha}{\partial x_\beta},$$
$$N_{\alpha\beta,1}=\frac{\partial N_{\alpha\beta}}{\partial x_1},$$
and
$$N_{\alpha\beta,2}=\frac{\partial N_{\alpha\beta}}{\partial x_2}.$$
\section{On the existence of a global minimizer}

At this point we present an existence result concerning the Kirchhoff-Love plate model.

We start with the following two remarks.
\begin{remark}\label{us123W}Let $\{P_\alpha\} \in L^\infty(\Omega;\mathbb{R}^2)$. We may easily obtain by appropriate Lebesgue integration $\{\tilde{T}_{\alpha \beta}\}$ symmetric and such that

 $$\tilde{T}_{\alpha\beta,\beta}=-P_\alpha, \text{ in } \Omega.$$

 Indeed, extending $\{P_\alpha\}$ to zero outside $\Omega$ if necessary, we may set
 $$\tilde{T}_{11}(x,y)=-\int_0^xP_1(\xi,y)\;d\xi,$$
  $$\tilde{T}_{22}(x,y)=-\int_0^yP_2(x,\xi)\;d\xi,$$
  and $$\tilde{T}_{12}(x,y)=\tilde{T}_{21}(x,y)=0, \text{ in } \Omega.$$

 Thus, we may choose a $C>0$ sufficiently big, such that $$\{T_{\alpha\beta}\}=\{\tilde{T}_{\alpha\beta} +C \delta_{\alpha\beta}\}$$ is positive definite $\text{ in } \Omega$, so that

 $$T_{\alpha\beta,\beta}=\tilde{T}_{\alpha\beta,\beta}=-P_\alpha,$$

 where $$\{\delta_{\alpha\beta}\}$$ is the Kronecker delta.

 So, for the kind of boundary conditions of the next theorem, we do NOT have any restriction for the $\{P_\alpha\}$ norm.

 Summarizing, the next result is new and it is really a step forward concerning the previous one in Ciarlet \cite{[3]}. We emphasize this result and
 its proof through such a tensor $\{T_{\alpha\beta}\}$ are new, even though the final part of the proof is established through a standard procedure in
 the calculus of variations.

 About the other existence result for plates, its proof through the tensor well specified $\{(T_0)_{\alpha\beta}\}$ is also new, even though the final  part
 of such a proof  is also performed through a standard procedure.

 A similar remark is valid for the existence result for the model of shells, which is also established through a tensor $T_0$ properly specified.

 Finally, the duality principles and concerning optimality conditions are established through new functionals. Similar results may be found
 in \cite{120}.
\end{remark}

\begin{remark} Specifically about the existence of the tensor $T_0$ relating Theorem \ref{us123a}, we recall the following well known duality principle
of the calculus of variations
\begin{eqnarray} &&\inf_{T=\{T_{\alpha\beta}\} \in B^*} \left\{\frac{1}{2}\|T\|_2^2 \right\}\nonumber \\ &=&
\sup_{ \{u_\alpha\} \in \tilde{U}} \left\{ -\frac{1}{2}\int_\Omega \nabla u_\alpha \cdot \nabla u_\alpha \;dx +\langle u_\alpha, P_\alpha\rangle_{L^2(\Omega)}
+\langle u_\alpha, P^t_\alpha\rangle_{L^2(\Gamma_t)}\right\}.
\end{eqnarray}
Here $$B^*=\{T \in L^2(\Omega;\mathbb{R}^4)\;:\; T_{\alpha\beta, \beta}+P_\alpha=0, \; \text{ in } \Omega,\; T_{\alpha\beta}n_\beta-P^t_\alpha=0, \text{ on } \Gamma_t\},$$
and
$$\tilde{U}=\{\{u_\alpha\} \in W^{1,2}(\Omega;\mathbb{R}^2)\;:\; u_\alpha=0 \text{ on } \Gamma_0\}.$$

We also recall that the existence of a unique solution for both these primal and dual convex formulations is a well known result of the duality theory in the calculus of variations. Please, see related results in \cite{[6]}.

A similar duality principle may be established for the case of Theorem \ref{us123b}.
\end{remark}
\begin{thm}\label{usp10}Let $\Omega \subset \mathbb{R}^2$ be an open, bounded, connected set with a Lipschitzian boundary denoted by $\partial \Omega=\Gamma.$ Suppose $(G \circ \Lambda):U \rightarrow \mathbb{R}$ is defined by
$$G(\Lambda u)=G_1(\gamma (u))+G_2(\kappa (u)), \; \forall u \in U,$$
where
$$G_1(\gamma u)= \frac{1}{2}\int_\Omega H_{\alpha\beta\lambda\mu} \gamma_{\alpha\beta}(u)\gamma_{\lambda \mu}(u)\;dx,$$
and
$$G_2(\kappa u)=\frac{1}{2}\int_\Omega h_{\alpha\beta\lambda\mu}\kappa_{\alpha\beta}(u)\kappa_{\lambda \mu}(u)\;dx,$$
where
$$\Lambda(u)=(\gamma(u),\kappa(u))=(\{\gamma_{\alpha\beta}(u)\},\{\kappa_{\alpha\beta}(u)\}),$$
$$\gamma_{\alpha\beta}(u)=\frac{u_{\alpha,\beta}+u_{\beta,\alpha}}{2}+\frac{w_{,\alpha} w_{,\beta}}{2},$$
$$\kappa_{\alpha\beta}(u)=-w_{,\alpha\beta},$$
and where
\begin{eqnarray}J(u)&=& W(\gamma(u),\kappa(u))-\langle P_\alpha, u_\alpha\rangle_{L^2(\Omega)}
\nonumber \\ &&-\langle w,P\rangle_{L^2(\Omega)}-\langle P^t_\alpha, u_\alpha \rangle_{L^2(\Gamma_t)} \nonumber \\&&
-\langle P^t,w \rangle_{L^2(\Gamma_t)},\end{eqnarray}
where,
 \begin{eqnarray} U&=&\{u=(u_\alpha,w)=(u_1,u_2,w) \in W^{1,2}(\Omega; \mathbb{R}^2) \times W^{2,2}(\Omega)\;:
\nonumber \\ && u_\alpha=w= \frac{\partial w}{\partial \mathbf{n}}=0, \text{ on } \Gamma_0\},
\end{eqnarray}
where $\partial \Omega =\Gamma_0 \cup \Gamma_t$ and the Lebesgue measures $$m_{\Gamma}(\Gamma_0 \cap \Gamma_t)=0,$$
and $$m_\Gamma (\Gamma_0)>0.$$

We also define,
\begin{eqnarray} F_1(u)&=& -\langle w,P\rangle_{L^2(\Omega)}-\langle u_\alpha, P_\alpha \rangle_{L^2(\Omega)}-\langle P^t_\alpha, u_\alpha \rangle_{L^2(\Gamma_t)} \nonumber \\&&
-\langle P^t,w \rangle_{L^2(\Gamma_t)}+\langle \varepsilon_\alpha, u_\alpha^2 \rangle_{L^2(\Gamma_t)} \nonumber \\
&\equiv& -\langle u,\mathbf{f}\rangle_{L^2}+\langle \varepsilon_\alpha, u_\alpha^2 \rangle_{L^2(\Gamma_t)}
\nonumber \\ &\equiv&  -\langle u,\mathbf{f}_1\rangle_{L^2}-\langle u_\alpha, P_\alpha \rangle_{L^2(\Omega)}+\langle \varepsilon_\alpha, u_\alpha^2 \rangle_{L^2(\Omega)} ,
\end{eqnarray}
where  $$\langle u,\mathbf{f}_1\rangle_{L^2}=\langle u,\mathbf{f}\rangle_{L^2}-\langle u_\alpha, P_\alpha \rangle_{L^2(\Omega)},$$ $\varepsilon_\alpha>0, \; \forall \alpha \in \{1,2\}$ and
$$\mathbf{f}=(P_\alpha,P) \in L^\infty(\Omega;\mathbb{R}^3).$$

Let $J:U \rightarrow \mathbb{R}$ be defined by
$$J(u)=G(\Lambda u)+F_1(u),\; \forall u \in U.$$
Assume there exists $\{c_{\alpha\beta}\} \in \mathbb{R}^{2 \times 2}$ such that $c_{\alpha\beta}>0,\; \forall \alpha,\beta \in \{1,2\}$ and $$G_2(\kappa(u)) \geq c_{\alpha\beta}\|w_{,\alpha\beta}\|_2^2,\; \forall u \in U.$$

Under such hypotheses, there exists $u_0 \in U$ such that
$$J(u_0)=\min_{u \in U} J(u).$$
\end{thm}
\begin{proof} Observe that we may find $\mathbf{T}_\alpha=\{(T_\alpha)_\beta\}$ such that
$$div \mathbf{T}_{\alpha}=T_{\alpha\beta,\beta}=-P_\alpha$$ an also such that $\{T_{\alpha\beta}\}$ is positive definite and symmetric (please, see Remark \ref{us123W}).

Thus defining \begin{equation}\label{a.1}v_{\alpha\beta}(u)= \frac{u_{\alpha,\beta}+u_{\beta,\alpha}}{2}+\frac{1}{2}w_{,\alpha}w_{,\beta},\end{equation} we obtain
\begin{eqnarray}\label{usp1}
J(u)&=&G_1(\{v_{\alpha\beta}(u)\})+G_2(\kappa(u))-\langle u,\mathbf{f}\rangle_{L^2}+\langle \varepsilon_\alpha, u_\alpha^2 \rangle_{L^2(\Gamma_t)} \nonumber \\ &=& G_1(\{v_{\alpha\beta}(u)\})+G_2(\kappa(u))+\langle T_{\alpha\beta,\beta},u_\alpha \rangle_{L^2(\Omega)}-\langle u,\mathbf{f}_1\rangle_{L^2}+\langle \varepsilon_\alpha, u_\alpha^2 \rangle_{L^2(\Gamma_t)}\nonumber \\ &=& G_1(\{v_{\alpha\beta}(u)\})+G_2(\kappa(u))-\left\langle T_{\alpha\beta},\frac{u_{\alpha,\beta}+u_{\beta,\alpha}}{2} \right\rangle_{L^2(\Omega)}
\nonumber \\ &&+\langle T_{\alpha\beta}n_\beta ,u_\alpha\rangle_{L^2(\Gamma_t)}-\langle u,\mathbf{f}_1\rangle_{L^2}+\langle \varepsilon_\alpha, u_\alpha^2 \rangle_{L^2(\Gamma_t)} \nonumber \\ &=& G_1(\{v_{\alpha\beta}(u)\})+G_2(\kappa(u))-\left\langle T_{\alpha\beta},v_{\alpha\beta}(u)-\frac{1}{2}w_{,\alpha}w_{,\beta} \right\rangle_{L^2(\Omega)}-\langle u,\mathbf{f}_1\rangle_{L^2}+\langle \varepsilon_\alpha, u_\alpha^2 \rangle_{L^2(\Gamma_t)} \nonumber \\ &&+\langle T_{\alpha\beta}n_\beta ,u_\alpha\rangle_{L^2(\Gamma_t)}\nonumber \\ &\geq& c_{\alpha\beta}\|w_{,\alpha\beta}\|_2^2+\frac{1}{2}\left\langle T_{\alpha\beta},w_{,\alpha}w_{,\beta} \right\rangle_{L^2(\Omega)}-\langle u,\mathbf{f}_1\rangle_{L^2}+\langle \varepsilon_\alpha, u_\alpha^2 \rangle_{L^2(\Gamma_t)}+G_1(\{v_{\alpha\beta}(u)\})\nonumber \\ &&-\langle T_{\alpha\beta},v_{\alpha\beta}(u)\rangle_{L^2(\Omega)}
+\langle T_{\alpha\beta}n_\beta ,u_\alpha\rangle_{L^2(\Gamma_t)}.
\end{eqnarray}

From this, since $\{T_{\alpha\beta}\}$ is positive definite,  clearly $J$ is bounded below.

Let $\{u_n\} \in U$ be a minimizing sequence for $J$. Thus there exists $\alpha_1 \in \mathbb{R}$ such that
$$\lim_{n \rightarrow \infty}J(u_n)= \inf_{u \in U} J(u)=\alpha_1.$$

From (\ref{usp1}), there exists $K_1>0$ such that
$$\|(w_n)_{,\alpha\beta}\|_2< K_1,\forall \alpha,\beta \in \{1,2\},\;n \in \mathbb{N}.$$

Therefore, there exists $w_0 \in W^{2,2}(\Omega)$ such that, up to a subsequence not relabeled,
$$(w_n)_{,\alpha\beta} \rightharpoonup (w_0)_{,\alpha\beta},\; \text{ weakly in } L^2,$$
 $\forall \alpha,\beta \in \{1,2\}, \text{ as } n \rightarrow \infty.$

Moreover, also up to a subsequence not relabeled,
\begin{equation}\label{a.2}(w_n)_{,\alpha} \rightarrow (w_0)_{,\alpha},\; \text{ strongly in } L^2  \text{ and } L^4,\end{equation}
 $\forall \alpha, \in \{1,2\}, \text{ as } n \rightarrow \infty.$

Also from (\ref{usp1}), there exists $K_2>0$ such that,
$$\|(v_n)_{\alpha\beta}(u)\|_2< K_2,\forall \alpha,\beta \in \{1,2\},\;n \in \mathbb{N},$$
and thus, from this, (\ref{a.1}) and (\ref{a.2}), we may infer that
there exists $K_3>0$ such that
$$\|(u_n)_{\alpha,\beta}+(u_n)_{\beta,\alpha}\|_2< K_3,\forall \alpha,\beta \in \{1,2\},\;n \in \mathbb{N}.$$

From this and Korn's inequality, there exists $K_4>0$ such that

$$\|u_n\|_{W^{1,2}(\Omega;\mathbb{R}^2)} \leq K_4, \; \forall n \in \mathbb{N}.$$
So, up to a subsequence not relabeled, there exists $\{(u_0)_\alpha\} \in W^{1,2}(\Omega, \mathbb{R}^2),$ such that
$$(u_n)_{\alpha,\beta}+(u_n)_{\beta,\alpha} \rightharpoonup (u_0)_{\alpha,\beta}+(u_0)_{\beta,\alpha},\; \text{ weakly in } L^2,$$
 $\forall \alpha,\beta \in \{1,2\}, \text{ as } n \rightarrow \infty,$
and,
$$(u_n)_{\alpha} \rightarrow (u_0)_{\alpha},\; \text{ strongly in } L^2,$$
 $\forall \alpha \in \{1,2\}, \text{ as } n \rightarrow \infty.$

Moreover, the boundary conditions satisfied by the subsequences are also satisfied for $w_0$ and $u_0$ in a trace sense, so that
 $$u_0=((u_0)_\alpha,w_0) \in U.$$

From this, up to a subsequence not relabeled, we get
$$\gamma_{\alpha\beta}(u_n) \rightharpoonup \gamma_{\alpha\beta}(u_0), \text{ weakly in } L^2,$$
$\forall \alpha,\beta \in \{1,2\},$
and
$$\kappa_{\alpha\beta}(u_n) \rightharpoonup \kappa_{\alpha\beta}(u_0), \text{ weakly in } L^2,$$
$\forall \alpha,\beta \in \{1,2\}.$

Therefore, from the convexity of $G_1$ in $\gamma$ and $G_2$ in $\kappa$ we obtain
\begin{eqnarray}
\inf_{u \in U}J(u)&=& \alpha_1 \nonumber \\ &=& \liminf_{n \rightarrow \infty}J(u_n) \nonumber \\ &\geq&
J(u_0).
\end{eqnarray}

Thus, $$J(u_0)=\min_{u\in U}J(u).$$

The proof is complete.
\end{proof}

\section{An analogous model in elasticity}

In this section we present similar results to those of previous sections for an elastic plate model.

The first analogous result is summarized by the following theorem.

\begin{thm}\label{TH20} Let $\Omega \subset \mathbb{R}^2$ be an open, bounded and connected set with a regular (Lipschitzian)
boundary denoted by $\partial \Omega.$

Consider the functional $J:U \rightarrow \mathbb{R}$ where
\begin{eqnarray}
J(u)&=& \frac{1}{2}\int_\Omega H_{\alpha\beta\lambda\mu}\gamma_{\alpha\beta}(u)\gamma_{\lambda \mu}(u)\;dx
\nonumber \\ &&+\frac{1}{2}\int_\Omega h_{\alpha\beta\lambda\mu}\kappa_{\alpha\beta}(u)\kappa_{\lambda\mu}(u)\;dx \nonumber \\
&&-\langle w,P \rangle_{L^2}-\langle u_\alpha, P_\alpha \rangle_{L^2},
\end{eqnarray}
$$u=(u_1,u_2,u_3)=(u_\alpha,w) \in U=W_0^{1,2}(\Omega;\mathbb{R}^2)\times W_0^{2,2}(\Omega).$$

Here $$\gamma_{\alpha\beta}(u)=\frac{u_{\alpha,\beta}+u_{\beta,\alpha}}{2}+\frac{1}{2}w_{,\alpha} w_{,\beta}$$
and
$$\kappa_{\alpha\beta}(u)=-w_{\alpha\beta}.$$

Moreover $P,P_\alpha \in C^1(\overline{\Omega})$ such that either $P>0$ or $P<0$ in $\Omega$. Moreover $\{H_{\alpha\beta\lambda\mu}\}$ is a constant symmetric fourth order tensor such that
$$H_{\alpha\beta\lambda\mu}t_{\alpha\beta}t_{\lambda \mu} \geq c_0 t_{\alpha\beta}t_{\alpha\beta},$$
$\forall \text{ symmetric } \{t_{\alpha\beta}\} \in \mathbb{R}^{2 \times 2},$
for some $c_0>0$.

Also, $$\{h_{\alpha\beta\lambda\mu}\}=\left\{ \frac{\hat{c}_1 H_{\alpha\beta\lambda\mu}}{12}\right\},$$ for some $\hat{c}_1>0.$

Suppose $c_1,c_2,c_3^\alpha \in \mathbb{R}^+$ are such that
\begin{eqnarray}&&\sqrt{\delta^2J(u,\varphi_\alpha,\varphi)+K \int_\Omega |\varphi|^2\;dx+K^\alpha \int_\Omega |\varphi_\alpha|^2\;dx}
\nonumber \\ &\geq& \int_\Omega (c_1+c_2|w|)|\varphi|\;dx+c_3^\alpha\int_\Omega |\varphi_\alpha|\;dx,
\end{eqnarray}
$\forall u \in U,\; (\varphi_\alpha,\varphi) \in C_c^\infty(\Omega;\mathbb{R}^3).$

Define also,
\begin{eqnarray}
A^+&=& \left\{u \in U\;:\; \int_\Omega (c_1+c_2|w|)|\varphi|\;dx+c_3^\alpha\int_\Omega |\varphi_\alpha|\;dx \right.
\nonumber \\ && \left.\geq \sqrt{K\int_\Omega |\varphi|^2\;dx+K^\alpha \int_\Omega |\varphi_\alpha|^2\;dx},
\; \forall (\varphi_\alpha,\varphi) \in C_c^\infty(\Omega;\mathbb{R}^3)\right\},\end{eqnarray}
$$B^+=\{u \in U\;:\; P w\geq 0, \text{ in } \overline{\Omega}\},$$
and
$$E=A^+ \cap B^+.$$

Under such hypotheses,
$$\inf_{u \in U}J(u)=\inf_{u \in B^+} J(u)$$ and
$E$ is convex.
\end{thm}
\begin{proof} Let $\alpha_1 \in \mathbb{R}$ be such that
$$\alpha_1=\inf_{u \in U}J(u).$$

Let $\varepsilon>0$. By density there exists $u_\varepsilon \in C^1(\overline{\Omega}; \mathbb{R}^3)\cap U$ such that
 $$\alpha_1 \leq J(u_\varepsilon) < \alpha_1+\varepsilon.$$

Define  \begin{equation}
\hat{w}_\varepsilon(x)=\left \{
\begin{array}{ll}
 w_\varepsilon(x), &  \text{ if }\; w_\varepsilon(x)P(x) \geq 0,
 \\
 -w_\varepsilon(x), &  \text{ if }\; w_\varepsilon(x)P(x) < 0,
  \end{array} \right.\end{equation}
$\forall x \in \overline{\Omega}.$

Hence $$\langle \hat{w}_\varepsilon,P\rangle_{L^2}\geq \langle w_\varepsilon,P\rangle_{L^2},$$
so that
$$J(\hat{u}_\varepsilon) \leq J(u_\varepsilon),$$
where
$$\hat{u}_\varepsilon=(u_\varepsilon^\alpha,\hat{w}_\varepsilon) \in B^+.$$

Hence, $$\alpha_1 \leq J(\hat{u}_\varepsilon) \leq J(u_\varepsilon)< \alpha_1+\varepsilon.$$

From this we obtain,

$$\alpha_1 \leq \inf_{u \in B^+}J(u) < \alpha_1+\varepsilon.$$

Since $\varepsilon>0$ is arbitrary, we may infer that
$$\inf_{u \in B^+}J(u)=\alpha_1=\inf_{u \in U}J(u).$$

Finally, let $u_1,u_2 \in E$ and $\lambda \in [0,1].$

Denoting
$$H(u,\varphi_\alpha,\varphi)=\int_\Omega (c_1+c_2|w|)|\varphi|\;dx+c_3^\alpha\int_\Omega |\varphi_\alpha|\;dx,$$
since $u_1,u_2 \in B^+$ we have that $$\text{ sign }w_1=\text{ sign }w_2, \text{ in } \overline{\Omega}.$$

Hence
\begin{eqnarray}
&&H(\lambda u_1+(1-\lambda)u_2,\varphi_\alpha,\varphi) \nonumber \\ &=& \lambda H(u_1,\varphi_\alpha,\varphi)+(1-\lambda)H(u_2,\varphi_\alpha,\varphi)  \nonumber \\ &\geq&
(\lambda+(1-\lambda))\sqrt{K\int_\Omega |\varphi|^2\;dx+K^\alpha \int_\Omega |\varphi_\alpha|^2\;dx} \nonumber \\ &=& \sqrt{K\int_\Omega |\varphi|^2\;dx+K^\alpha \int_\Omega |\varphi_\alpha|^2\;dx},\; \forall \mathbf{\varphi} \in C^\infty_c(\Omega;\mathbb{R}^3).
\end{eqnarray}

From this we may easily infer that $$\lambda u_1+(1-\lambda)u_2 \in A^+ \cap B^+=E,$$ so that $E$ is convex.

The proof is complete.
\end{proof}

\begin{thm} Let $\Omega \subset \mathbb{R}^2$ be an open, bounded and connected set with a regular (Lipschitzian)
boundary denoted by $\partial \Omega.$

Consider the functional $J:U \rightarrow \mathbb{R}$ where
\begin{eqnarray}
J(u)&=& \frac{1}{2}\int_\Omega H_{\alpha\beta\lambda\mu}\gamma_{\alpha\beta}(u)\gamma_{\lambda \mu}(u)\;dx
\nonumber \\ &&+\frac{1}{2}\int_\Omega h_{\alpha\beta\lambda\mu}\kappa_{\alpha\beta}(u)\kappa_{\lambda\mu}(u)\;dx \nonumber \\
&&-\langle w,P \rangle_{L^2}-\langle u_\alpha, P_\alpha \rangle_{L^2},
\end{eqnarray}
$$u=(u_1,u_2,u_3)=(u_\alpha,w) \in U=W_0^{1,2}(\Omega;\mathbb{R}^2)\times W_0^{2,2}(\Omega).$$

Here $$\gamma_{\alpha\beta}(u)=\frac{u_{\alpha,\beta}+u_{\beta,\alpha}}{2}+\frac{1}{2}w_{,\alpha} w_{,\beta}$$
and
$$\kappa_{\alpha\beta}(u)=-w_{\alpha\beta}.$$

We also denote $Y_1=Y_1^*=L^2(\Omega;\mathbb{R}^{2 \times 2})$, $Y_2=Y_2^*=L^2(\Omega;\mathbb{R}^2)$
where generically, $$N=\{N_{\alpha\beta}\} \in Y_1^* \text{ and } Q=\{Q_\alpha\} \in Y_2^*.$$

Moreover $P,P_\alpha \in C^1(\overline{\Omega})$ and $\{H_{\alpha\beta\lambda\mu}\}$ is a constant symmetric fourth order tensor such that
$$H_{\alpha\beta\lambda\mu}t_{\alpha\beta}t_{\lambda \mu} \geq c_0 t_{\alpha\beta}t_{\alpha\beta},$$
$\forall \text{ symmetric } \{t_{\alpha\beta}\} \in \mathbb{R}^{2 \times 2},$
for some $c_0>0$.

Also, $$\{h_{\alpha\beta\lambda\mu}\}=\left\{ \frac{\hat{c}_1 H_{\alpha\beta\lambda\mu}}{12}\right\},$$ for some $\hat{c}_1>0.$

Define also,
$$B^*_0=\{\{Q_\alpha\} \in Y^*_2\;:\; \|Q\|_\infty  \leq K_2\},$$
\begin{eqnarray}B^*&=&\{Q \in B_0^* \text{ such that there exists }\hat{u} \in U_1^0\cap U_2
\nonumber \\ && \text{ such that } F^*(Q)=\langle \hat{w}_{,\alpha},Q_\alpha\rangle_{L^2}-F(\hat{u})\}
\end{eqnarray}
where $K_2>0$ is such that
$$\frac{-8K_2^2}{K}\{\delta_{\alpha\beta}\}+\{\overline{H}_{\alpha\beta\lambda\mu}\} > \mathbf{0}$$ in an appropriate tensor sense,
and where
$$\{\overline{H}_{\alpha\beta\lambda\mu}\}=\{H_{\alpha\beta\lambda\mu}\}^{-1}.$$

Moreover define,
$$B^*_1=\{\{N_{\alpha\beta}\} \in Y_1^*\;:\; \{-N_{\alpha\beta}+K \delta_{\alpha\beta}\} > K \delta_{\alpha\beta}/2\},$$

$$B_2^*=\{\{N_{\alpha\beta}\} \in Y_1^*\;:\; N_{\alpha\beta,\beta}+P_\alpha=0, \text{ in } \Omega\},$$
and
$C^*=B^*_1 \cap B_2^*$.

For each $u \in U$, define

\begin{eqnarray}L_1(u)&=&\sup_{N \in Y_1^*}\left\{ \frac{1}{2}\int_\Omega N_{\alpha\beta}w_{,\alpha}w_{,\beta}\right. \nonumber \\ && \left.-\frac{1}{2}\int_\Omega \overline{H}_{\alpha\beta\lambda\mu}
N_{\alpha\beta}N_{\lambda\mu}\;dx-\langle u_{\alpha}, N_{\alpha\beta,\beta}+P_\alpha \rangle_{L^2}
\right\},
\end{eqnarray}
\begin{eqnarray}L_2(u)&=&\sup_{N \in B^*}\left\{ \frac{1}{2}\int_\Omega N_{\alpha\beta}w_{,\alpha}w_{,\beta}\right. \nonumber \\ && \left.-\frac{1}{2}\int_\Omega \overline{H}_{\alpha\beta\lambda\mu}
N_{\alpha\beta}N_{\lambda\mu}\;dx-\langle u_{\alpha}, N_{\alpha\beta,\beta}+P_\alpha \rangle_{L^2}
\right\},
\end{eqnarray}
and
$$U_1=\{u \in U\;:\; \|u\|_{2,\infty} \leq \sqrt[4]{K} \text{ and } L_1(u)=L_2(u)\}.$$

Furthermore, define
$F:U_1 \rightarrow \mathbb{R}$ by
\begin{eqnarray}
F(u)&=& \frac{1}{2}\int_\Omega h_{\alpha\beta\lambda\mu}\kappa_{\alpha\beta}(u)\kappa_{\lambda\mu}(u)\;dx
\nonumber \\ &&+\frac{K}{2}\int_\Omega w_{,\alpha}w_{,\alpha}\;dx-\langle w,P \rangle_{L^2},
\end{eqnarray}
$G:U_1 \rightarrow \mathbb{R}$ by
\begin{eqnarray}
G(u)&=& \frac{1}{2}\int_\Omega H_{\alpha\beta\lambda\mu}\gamma_{\alpha\beta}(u)\gamma_{\lambda\mu}(u)\;dx
\nonumber \\ &&+\frac{K}{2}\int_\Omega w_{,\alpha}w_{,\alpha}\;dx,
\end{eqnarray}
$F^*:Y^*_2 \rightarrow \mathbb{R}$ by
\begin{eqnarray}
F^*(Q)&=& \sup_{ u \in U}\{\langle w_{,\alpha},Q_\alpha \rangle_{L^2}-F(u)\}\end{eqnarray}
and $G^*:Y_1^*\times Y_2^* \rightarrow \mathbb{R}$ by
\begin{eqnarray}
G^*(Q,N)&=&\sup_{ v_2 \in Y_2}\{\inf_{v_1 \in Y_1}\{\langle (v_2)_\alpha,Q_\alpha \rangle_{L^2}+
\langle (v_1)_{\alpha\beta},N_{\alpha\beta} \rangle_{L^2}
\nonumber \\ &&+\frac{1}{2}\int_\Omega H_{\alpha\beta\lambda\mu}
[(v_1)_{\alpha\beta}+\frac{1}{2} (v_2)_\alpha (v_2)_\beta][(v_1)_{\lambda\mu}+\frac{1}{2} (v_2)_\lambda (v_2)_\mu]
\;dx\nonumber \\ &&-\frac{K}{2}\int_\Omega (v_2)_\alpha (v_2)_\beta\;dx\}\}
\nonumber \\ &=& \frac{1}{2}\int_\Omega \overline{N_{\alpha\beta}^K}Q_\alpha Q_\beta\;dx \nonumber \\ &&
+\frac{1}{2}\int_\Omega \overline{H}_{\alpha\beta\lambda\mu}N_{\alpha\beta}N_{\lambda \mu}\;dx,
\end{eqnarray}
if $\{N_{\alpha\beta}\} \in B^*.$

Here
\begin{equation}
\{\overline{N_{\alpha\beta}^K}\}=\{N_{\alpha\beta}-K\delta_{\alpha\beta}\}^{-1}.\end{equation}

Finally define, $$J^*(Q,N)=-F^*(Q)+G^*(Q,N),$$

Suppose $c_1,c_2,c_3^\alpha \in \mathbb{R}^+$ are such that
\begin{eqnarray}&&\sqrt{\delta^2J(u,\varphi_\alpha,\varphi)+K \int_\Omega |\varphi|^2\;dx+K^\alpha \int_\Omega |\varphi_\alpha|^2\;dx}
\nonumber \\ &\geq& \int_\Omega (c_1+c_2|w|)|\varphi|\;dx+c_3^\alpha\int_\Omega |\varphi_\alpha|\;dx,
\end{eqnarray}
$\forall u \in U,\; (\varphi_\alpha,\varphi) \in C_c^\infty(\Omega;\mathbb{R}^3).$

Define also
\begin{eqnarray}
A^+&=& \left\{u \in U\;:\; \int_\Omega (c_1+c_2|w|)|\varphi|\;dx+c_3^\alpha\int_\Omega |\varphi_\alpha|\;dx \right.
\nonumber \\ && \left.\geq \sqrt{K\int_\Omega |\varphi|^2\;dx+K^\alpha \int_\Omega |\varphi_\alpha|^2\;dx},
\; \forall (\varphi_\alpha,\varphi) \in C_c^\infty(\Omega;\mathbb{R}^3)\right\},\end{eqnarray}
$$B^+=\{u \in U\;:\; P w\geq 0, \text{ in } \overline{\Omega}\},$$
and
$$E=A^+ \cap B^+.$$

Let $u_0 \in U$ be such that $u_0 \in U_1 \cap E$ and $\delta J(u_0)=\mathbf{0}.$

Defining
$$(N_0)_{\alpha\beta}=H_{\alpha\beta\lambda\mu}\left(\frac{(u_0)_{\lambda,\mu}+(u_0)_{\mu,\lambda}}{2}+\frac{1}{2}(w_0)_{,\alpha}
(w_0)_{,\beta}\right),$$

$$(Q_0)_\alpha=-(N_0)_{\alpha\beta}(w_0)_\beta+K (w_0)_\alpha$$
suppose
$$(Q_0,N_0) \in B^*\times C^*.$$

Under such hypotheses, $E$ is convex and
\begin{eqnarray}
J(u_0)&=&\inf_{u \in E} J(u) \nonumber \\ &=&
\inf_{u \in U}J(u) \nonumber \\ &=& \inf_{Q \in B^*}\sup_{N \in C^*}J^*(Q,N)
\nonumber \\ &=& J^*(Q_0,N_0).
\end{eqnarray}

\end{thm}
\begin{proof} Define $$\alpha_1=\inf_{u \in U} J(u).$$

Hence
\begin{eqnarray}
\alpha_1  &\leq&
-\langle w_{,\alpha}\,Q_\alpha \rangle_{L^2}+F(u) \nonumber \\ && + \sup_{N \in C^*}\{
\langle w_{,\alpha}\,Q_\alpha \rangle_{L^2}
\nonumber \\ &&-\langle u_\alpha, N_{\alpha\beta,\beta} +P_\alpha\rangle_{L^2} +\frac{1}{2}\int_\Omega N_{\alpha\beta} w_{,\alpha}w_{,\beta}\;dx \nonumber \\ &&-\frac{K}{2}\int_\Omega w_{,\alpha}w_{,\alpha}\;dx
-\frac{1}{2}\int_\Omega \overline{H}_{\alpha\beta\lambda\mu}N_{\alpha\beta}N_{\lambda \mu}\;dx\}
\nonumber \\ &\leq& -\langle w_{,\alpha}\,Q_\alpha \rangle_{L^2}+F(u) \nonumber \\ && \sup_{N \in C^*}\left\{\sup_{v_2 \in Y_2}\left\{\langle (v_2)_\alpha\,Q_\alpha \rangle_{L^2} +\frac{1}{2}\int_\Omega N_{\alpha\beta} (v_2)_\alpha (v_2)_\beta\;dx \right.\right. \nonumber \\ &&\left.\left.-\frac{K}{2}\int_\Omega (v_2)_\alpha (v_2)_\alpha\;dx
-\frac{1}{2}\int_\Omega \overline{H}_{\alpha\beta\lambda\mu}N_{\alpha\beta}N_{\lambda \mu}\;dx\right\}\right\}
\nonumber \\ &=&-\langle w_{,\alpha}\,Q_\alpha \rangle_{L^2}+F(u) \nonumber \\ && +\sup_{N \in C^*}\left\{-\frac{1}{2}\int_\Omega \overline{N_{\alpha\beta}^K}Q_\alpha Q_\beta\;dx\right. \nonumber \\ &&
\left.-\frac{1}{2}\int_\Omega \overline{H}_{\alpha\beta\lambda\mu}N_{\alpha\beta}N_{\lambda \mu}\;dx\right\},
\nonumber\end{eqnarray}
$\forall u \in U_1, Q \in B^*.$

Thus,
\begin{eqnarray}
\alpha_1 &\leq& \inf_{ u \in U_1}\{-\langle w_{,\alpha}\,Q_\alpha \rangle_{L^2}+F(u)\} \nonumber \\ && +\sup_{N \in C^*}\left\{-\frac{1}{2}\int_\Omega \overline{N_{\alpha\beta}^K}Q_\alpha Q_\beta\;dx \right.\nonumber \\ && \left.
-\frac{1}{2}\int_\Omega \overline{H}_{\alpha\beta\lambda\mu}N_{\alpha\beta}N_{\lambda \mu}\;dx \right\}\nonumber \\ &=&
-F^*(Q^*)+\sup_{ N \in C^*}G^*(Q,N) \nonumber \\ &=&
\sup_{N \in C^*} J^*(Q,N),
\end{eqnarray}
$\forall Q \in B^*.$

Summarizing, we have got
\begin{equation}\label{uk10010} \alpha_1=\inf_{u \in U}J(u) \leq \inf_{Q \in B^*}\left\{ \sup_{N \in C^*} J^*(Q,N)\right\}. \end{equation}

Similarly as in the proof of the Theorem \ref{TH1}, we may obtain
$$\delta J^*(Q_0,N_0)-\langle (u_0)_\alpha, (N_0)_{\alpha\beta,\beta} +P_\alpha\rangle_{L^2} =\mathbf{0},$$
\begin{equation}\label{us398}J^*(Q_0,N_0)=J(u_0),\end{equation}
and
\begin{equation}\label{us598}J^*(Q_0,N_0)=\sup_{N \in C^*} J^*(Q_0,N).\end{equation}

From the proof of Theorem \ref{TH20} we may infer that $E$ is convex.

From this, since $u_0 \in U_1 \cap E$ and $\delta J(u_0)=\mathbf{0}$ we have that, also from the Theorem \ref{TH20},
$$J(u_0)=\inf_{u \in E}J(u)=\inf_{u \in U}J(u).$$

Consequently, from such a result, from $Q_0 \in B^*$, (\ref{uk10010}), (\ref{us398}) and (\ref{us598}) we have that
\begin{eqnarray}
J(u_0)&=&\inf_{u \in E} J(u) \nonumber \\ &=&
\inf_{u \in U}J(u) \nonumber \\ &=& \inf_{Q \in B^*}\sup_{N \in C^*}J^*(Q,N)
\nonumber \\ &=& J^*(Q_0,N_0).
\end{eqnarray}
The proof is complete.
\end{proof}

\section{An auxiliary theoretical result in analysis}

In this section we state and prove some theoretical results in analysis which will be used in the
subsequent sections.

\begin{thm} Let $\Omega \subset \mathbb{R}^3$ be an open, bounded and connected set with a regular (Lipschitzian)
boundary denoted by $\partial \Omega.$

Assume $\{u_n\} \subset W^{1,4}(\Omega)$ be such that $$\|u_n\|_{1,4} \leq K, \; \forall \in \mathbb{N},$$
for some $K>0.$

Under such hypotheses there exists $u_0 \in W^{1,4}(\Omega) \cap C(\overline{\Omega})$ such that, up to a not relabeled subsequence,
$$u_n \rightharpoonup u_0, \text{ weakly in } W^{1,4}(\Omega),$$
$$u_n \rightarrow u_0 \text{ uniformly in } \overline{\Omega}$$
and
$$u_n \rightarrow u_0, \text{ strongly in } W^{1,3}(\Omega).$$
\end{thm}
\begin{proof} Since $W^{1,4}(\Omega)$ is reflexive, from the Kakutani and Sobolev Imbedding  theorems, up to a not relabeled
there exists $u_0 \in W^{1,4}(\Omega)$ such that
$$u_n \rightharpoonup u_0, \text{ weakly in } W^{1,4}(\Omega),$$
and
$$u_n \rightarrow u_0, \text{ strongly in } L^4(\Omega).$$

From the Rellich-Kondrachov Theorem, since for $m=1$, $p=4$ and $n=3$, we have $mp>n$, the following imbedding
is compact,
$$W^{1,4}(\Omega) \hookrightarrow C(\overline{\Omega}).$$

Thus, $$\{u_n\} \subset C(\overline{\Omega}),$$ and again up to a not relabeled subsequence,
$$u_n \rightarrow u_0 \text{ uniformly in } \overline{\Omega},$$
and also $$u_0 \in C(\overline{\Omega}),$$ so that
$$u_0 \in W^{1,4}(\Omega) \cap C(\overline{\Omega}).$$

Let $\varepsilon>0$. Hence, there exists $n_0 \in \mathbb{N}$ such that if $n>n_0$, then
$$|u_n(x)-u_0(x)|< \varepsilon, \text{ for almost all } x \in \Omega.$$

Let $$\varphi \in C^1_c(\Omega).$$ Choose $j \in \{1,2,3\}.$

Therefore, we may obtain

\begin{eqnarray}
\left|\left\langle \frac{\partial u_n }{\partial x_j}-\frac{\partial u_0 }{\partial x_j}, \varphi \right\rangle_{L^2}\right|
&=& \left|\left\langle u_n-u_0, \frac{\partial \varphi }{\partial x_j} \right\rangle_{L^2}\right|
\nonumber \\ &\leq&  \left\langle |u_n-u_0|, \left|\frac{\partial \varphi }{\partial x_j}\right| \right\rangle_{L^2}
\nonumber \\ &\leq& \varepsilon \left\|\frac{\partial \varphi }{\partial x_j}\right\|_1,\; \forall n >n_0.\end{eqnarray}

From this we may infer that
$$\lim_{n \rightarrow \infty}\left\langle \frac{\partial u_n }{\partial x_j}-\frac{\partial u_0 }{\partial x_j}, \varphi \right\rangle_{L^2}=0, \forall \varphi \in C^1_c(\Omega).$$

At this point we claim that
$$\lim_{n \rightarrow \infty}\left\langle \frac{\partial u_n }{\partial x_j}-\frac{\partial u_0 }{\partial x_j}, \varphi \right\rangle_{L^2}=0, \forall \varphi \in C_c(\Omega).$$

To prove such a claim, let $\varphi \in C_c(\Omega).$

Let a new $\varepsilon>0$ be given.

Hence, for each $r>0$ there exists $n_r \in \mathbb{N}$ such that if $n>n_r$, then $$\|u_n-u_0\|_\infty < \varepsilon r.$$

Observe that by density, we may obtain $\varphi_1 \in C^1_c(\Omega)$ such that
$$\|\varphi-\varphi_1\|_\infty< \varepsilon.$$

Hence,
\begin{eqnarray}
&&\left|\left\langle \frac{\partial u_n }{\partial x_j}-\frac{\partial u_0 }{\partial x_j}, \varphi \right\rangle_{L^2}\right|
\nonumber \\ &\leq& \left\langle \left|\frac{\partial u_n }{\partial x_j}-\frac{\partial u_0 }{\partial x_j}\right|, |\varphi -\varphi_1|\right\rangle_{L^2} \nonumber \\ &&+  \left|\left\langle\frac{\partial u_n }{\partial x_j}-\frac{\partial u_0 }{\partial x_j}, \varphi_1 \right\rangle_{L^2}\right| \nonumber \\ &\leq&
 \int_\Omega \left|\frac{\partial u_n }{\partial x_j}-\frac{\partial u_0 }{\partial x_j}\right|\;dx\|\varphi-\varphi_1\|_\infty \nonumber \\ &&+\left|\left\langle u_n-u_0,\frac{\partial \varphi_1 }{\partial x_j} \right\rangle_{L^2}\right| \nonumber \\ &\leq& 2 K_1 \varepsilon +\varepsilon\left\|\frac{\partial  \varphi_1}{\partial x_j}\right\|_1 \nonumber \\ &=& \left(2K_1+\left\|\frac{\partial  \varphi_1}{\partial x_j}\right\|_1\right)\varepsilon, \forall n>n_1.
 \end{eqnarray}
where $K_1>0$ is such that $$\|u_0\|_1<K_1, \|u_0\|_2<K_1$$ and $$\|u_n\|_1 < K_1,\;\|u_n\|_2 < K_1\; \forall n \in \mathbb{N}.$$

From this we may infer that
\begin{equation}\label{eq29}\lim_{n \rightarrow \infty}\left\langle \frac{\partial u_n }{\partial x_j}-\frac{\partial u_0 }{\partial x_j}, \varphi \right\rangle_{L^2}=0,\; \forall \varphi \in C_c(\Omega)\end{equation} so that
the claim holds.

Since $\Omega$ is bounded, we have $W^{1,4}(\Omega) \subset W^{1,2}(\Omega)$.

From the Gauss-Green Formula for such a latter space, we obtain

\begin{eqnarray}
&&\lim_{n \rightarrow \infty }\left|\frac{\partial u_n(x) }{\partial x_j}-\frac{\partial u_0(x) }{\partial x_j}\right|
\nonumber \\ &=&\lim_{ n \rightarrow \infty }\left( \lim_{r \rightarrow 0^+} \frac{\left| \int_{B_r(x)}\left( \frac{\partial u_n(y) }{\partial x_j}-\frac{\partial u_0(y) }{\partial x_j}\right)\;dy\right|}{m(B_r(x))}\right) \nonumber \\ &\leq&\limsup_{ n \rightarrow \infty }\left( \limsup_{r \rightarrow 0^+} \frac{\left| \int_{B_r(x)}\left( \frac{\partial u_n(y) }{\partial x_j}-\frac{\partial u_0(y) }{\partial x_j}\right)\;dy\right|}{m(B_r(x))}\right) \nonumber \\ &=&\limsup_{r \rightarrow 0^+}\left(\limsup_{ n \rightarrow \infty }  \frac{\left| \int_{B_r(x)}\left( \frac{\partial u_n(y) }{\partial x_j}-\frac{\partial u_0(y) }{\partial x_j}\right)\;dy\right|}{m(B_r(x))}\right) \nonumber \\ &=&
\limsup_{r \rightarrow 0^+}\left(\limsup_{ n \rightarrow \infty }\frac{\left| \int_{\partial B_r(x)} (u_n(y)-u_0(y))\nu_i \;dS(y)\right|}{m(B_r(x))}\right) \nonumber \\ &=&
\limsup_{r \rightarrow 0^+}\left(\limsup_{ n \rightarrow \infty }\frac{\left| (u_n(\tilde{y})-u_0(\tilde{y}))\int_{\partial B_r(x)} \nu_i \;dS(y)\right|}{m(B_r(x))}\right) \nonumber \\ &\leq& \varepsilon \limsup_{r \rightarrow 0^+} \frac{ \int_{\partial B_r(x)} r |\nu_i| \;dS(y)}{m(B_r(x))}
 \nonumber \\ &\leq& K_1 \varepsilon, \text{ for almost all } x \in \Omega,\end{eqnarray}
where $\tilde{y} \in \overline{B_r(x)}$ depends on $r$ and $n$.

Therefore, we may infer that
$$\lim_{n \rightarrow \infty}\frac{\partial u_n(x) }{\partial x_j}=\frac{\partial u_0(x) }{\partial x_j}, \text{ a. e. in }  \Omega.$$

Here we define
$$A_{n, \varepsilon}=\left\{ x \in \Omega \;:\;\left|\frac{\partial u_n(x) }{\partial x_j}-\frac{\partial u_0(x) }{\partial x_j}\right|< \varepsilon\right\}.$$

Define also
$$B_n=\cap_{k=n}^\infty A_{k,\varepsilon}.$$

Observe that for almost all $x \in \Omega$, there exists $n_x \in \mathbb{N}$ such that if $n>n_x$, then
$$\left|\frac{\partial u_n(x) }{\partial x_j}-\frac{\partial u_0(x) }{\partial x_j}\right|< \varepsilon,$$
so that almost all $x \in B_n,\; \forall n>n_x.$

From this $$\Omega = \left(\cup_{n=1}^\infty B_n\right) \cup B_0,$$
where $m(B_0)=0.$

Also $$\cup_{k=1}^n B_k=B_n,$$
so that
$$\lim_{ n \rightarrow \infty } m(B_n)=m(\Omega).$$

Observe that there exists $n_0 \in \mathbb{N}$ such that if $n>n_0$, then
$$\sqrt[4]{m(\Omega \setminus B_n)}< \varepsilon/K^3.$$

Consequently fixing $n >n_0$, from the generalized H\"{o}lder inequality, if $m>n$, we have
\begin{eqnarray}&&\int_\Omega \left|\frac{\partial u_m }{\partial x_j}-\frac{\partial u_0 }{\partial x_j}\right|^3 \;dx
\nonumber \\ &=& \int_{\Omega \setminus B_n} \left|\frac{\partial u_m }{\partial x_j}-\frac{\partial u_0 }{\partial x_j}\right|^3\;dx
\nonumber \\ &&+ \int_{B_n}\left| \frac{\partial u_m }{\partial x_j}-\frac{\partial u_0 }{\partial x_j}\right|^3\;dx \nonumber \\ &\leq& \left\|\frac{\partial u_m }{\partial x_j}-\frac{\partial u_0 }{\partial x_j}\right\|_4^3 \|\chi_{\Omega\setminus B_n}\|_4 +\varepsilon^3 m(\Omega) \nonumber \\ &\leq&
\varepsilon+ \varepsilon^3 m(\Omega).
\end{eqnarray}

Summarizing, we may infer that
$$\int_\Omega \left|\frac{\partial u_m }{\partial x_j}-\frac{\partial u_0 }{\partial x_j}\right|^3 \;dx \rightarrow 0, \text{ as } m \rightarrow \infty, \; \forall j \in \{1,2,3\}.$$
so that
$$u_n \rightarrow u_0, \text{ strongly in } W^{1,3}(\Omega).$$

The proof is complete.

\end{proof}
\section{An existence result for a model in elasticity}
In this section we present an existence result for a non-linear elasticity model.

\begin{thm}
Let $\Omega \subset \mathbb{R}^3$ be an open, bounded and connected set with a regular (Lipschitzian)
boundary denoted by $\partial \Omega.$

Consider the functional $J:U \rightarrow \mathbb{R}$ defined by
\begin{eqnarray}
J(u)&=& \frac{1}{2}\int_\Omega H_{ijkl}\left( \frac{u_{i,j}+u_{j,i}}{2}+\frac{u_{m,i}u_{m,j}}{2}\right)\left( \frac{u_{k,l}+u_{l,k}}{2}+\frac{u_{p,k}u_{p,l}}{2}\right)\;dx \nonumber \\ &&-\langle P_i,u_i \rangle_{L^2},
\end{eqnarray}
where $U=W_0^{1,4}(\Omega;\mathbb{R}^3)$, $P_i \in L^\infty(\Omega),\;\forall i \in \{1,2,3\}.$

Moreover, $\{H_{ijkl}\}$ is a fourth order constant tensor such that
$$H_{ijkl}t_{ij}t_{kl} \geq c_0 t_{ij}t_{ij}, \; \forall \text{ symmetric } t \in \mathbb{R}^{2 \times 2}$$
and
$$H_{ijkl}t_{mi}t_{mj}t_{kp}t_{lp} \geq c_1 \sum_{i,j=1}^3 t_{ij}^4,\; \forall \text{ symmetric } t \in \mathbb{R}^{2\times 2},$$
for some real constants $c_0>0, c_1>0$.

Under such hypotheses, there exists $u_0 \in U$ such that
$$J(u_0)=\min_{u \in U} J(u).$$
\end{thm}
\begin{proof}
First observe that we may find a positive definite tensor $\{T_{ij}\} \subset L^\infty(\Omega;\mathbb{R}^{2\times 2})$ such that
$$T_{ij,j}+P_i=0, \text{ in } \Omega.$$

Hence, denoting
$$v_{ij}(u)=\frac{u_{i,j}+u_{j,i}}{2}+\frac{u_{m,i}u_{m,j}}{2},$$
we have
$$\frac{u_{i,j}+u_{j,i}}{2}=v_{ij}(u)-\frac{u_{m,i}u_{m,j}}{2},$$
so that
\begin{eqnarray}
J(u)&=&\frac{1}{2}\int_\Omega H_{ijkl}v_{ij}(u)v_{kl}(u)\;dx+\langle T_{ij,j},u_i\rangle_{L^2} \nonumber \\ &=&
\frac{1}{2}\int_\Omega H_{ijkl}v_{ij}(u)v_{kl}(u)\;dx-\left\langle T_{ij},\frac{u_{i,j}+u_{j,i}}{2}\right\rangle_{L^2} \nonumber \\ &=&\frac{1}{2}\int_\Omega H_{ijkl}v_{ij}(u)v_{kl}(u)\;dx-\left\langle T_{ij},v_{ij}(u)-\frac{u_{mi}\;u_{mj}}{2}\right\rangle_{L^2} \nonumber \\ &=&\frac{1}{2}\int_\Omega H_{ijkl}v_{ij}(u)v_{kl}(u)\;dx-\left\langle T_{ij},v_{ij}(u)\right\rangle_{L^2}+\left\langle T_{ij},\frac{u_{mi}\;u_{mj}}{2}\right\rangle_{L^2},\; \forall u \in U.
\end{eqnarray}
From this and the hypotheses on $\{H_{ijkl}\}$ it is clear that $J$ is bounded below so that there exists $\alpha \in \mathbb{R}$ such that
$$\alpha=\inf_{u \in U} J(u).$$

Let $\{u_n\} \subset U$ be a minimizing sequence for $J$, that is, let such a sequence be such that
$$J(u_n) \rightarrow \alpha, \; \text{ as } n \rightarrow \infty.$$

Also from the hypotheses on $\{H_{ijkl}\}$ and the Poincar\'{e} inequality, we have that there exists $K>0$ such that
$$ \|u_n\|_{1,4} \leq K,\; \forall n \in \mathbb{N}.$$

From the auxiliary result in the last section, there exists $u_0 \in C^0(\overline{\Omega};\mathbb{R}^3) \cap W^{1,4}(\Omega;\mathbb{R}^3)$ such that, up to a not relabeled subsequence,
$$u_n \rightarrow u_0, \text{ strongly in } W^{1,3}(\Omega:\mathbb{R}^3).$$

From such a latter result, up to a not relabeled subsequence, we may obtain
$$\frac{(u_n)_{i,j}+(u_n)_{j,i}}{2}+\frac{(u_n)_{m,i}(u_n)_{m,j}}{2} \rightharpoonup \frac{(u_0)_{i,j}+(u_0)_{j,i}}{2}+\frac{(u_0)_{m,i}(u_0)_{m,j}}{2}, \text{ weakly in } L^{3/2}(\Omega).$$

Since $L^{3/2}(\Omega)$ is reflexive, from the convexity of $J$ in $v_{ij}(u)$ and since $\{T_{ij}\}$ is positive definite, we have that
$$\alpha=\liminf_{n \rightarrow \infty}J(u_n) \geq J(u_0),$$
so that
$$J(u_0)=\min_{u \in U} J(u).$$

The proof is complete.

\end{proof}

\section{Another existence result for a model in elasticity}
In this section we present another existence result  for a similar (to the previous one)  non-linear elasticity model.

\begin{thm}
Let $\Omega \subset \mathbb{R}^3$ be an open, bounded and connected set with a regular (Lipschitzian)
boundary denoted by $\partial \Omega.$

Consider the functional $J:U \rightarrow \mathbb{R}$ defined by
\begin{eqnarray}
J(u)&=& \frac{1}{2}\int_\Omega H_{ijkl}\left( \frac{u_{i,j}+u_{j,i}}{2}+\frac{u_{m,i}u_{m,j}}{2}\right)\left( \frac{u_{k,l}+u_{l,k}}{2}+\frac{u_{p,k}u_{p,l}}{2}\right)\;dx \nonumber \\ &&-\langle P_i,u_i \rangle_{L^2}-\langle P^t_i,u_i \rangle_{L^2(\Gamma_t)},
\end{eqnarray}
where $$\partial \Omega =\Gamma=\Gamma_0 \cup \Gamma_t,$$ $$\Gamma_0 \cap \Gamma_t=\emptyset,$$
$m_\Gamma(\Gamma_0)>0,$ $m_\Gamma(\Gamma_t)>0,$ $P_i \in L^\infty(\Omega)\cap W^{1,2}(\Omega),\; P^t_i \in L^\infty(\Gamma_t),\;
\forall i \in \{1,2,3\}.$

Moreover $$U=\{u \in W^{1,4}(\Omega;\mathbb{R}^3)\;: \; u=\hat{u}_0 \text{ on } \Gamma_0\},$$
where we assume $\hat{u}_0 \in W^{1,4}(\Omega).$

Furthermore, $\{H_{ijkl}\}$ is a fourth order symmetric constant tensor such that
$$H_{ijkl}t_{ij}t_{kl} \geq c_0 t_{ij}t_{ij}, \; \forall \text{ symmetric } t \in \mathbb{R}^{2 \times 2}$$
and
$$H_{ijkl}t_{mi}t_{mj}t_{kp}t_{lp} \geq c_1 \sum_{i,j=1}^3 t_{ij}^4,\; \forall \text{ symmetric } t \in \mathbb{R}^{2\times 2},$$
for some real constants $c_0>0, c_1>0$.

Under such hypotheses, there exists $u_0 \in U$ such that
$$J(u_0)=\min_{u \in U} J(u).$$
\end{thm}
\begin{proof}
First observe that we may find a positive definite tensor $\{T_{ij}\} \subset L^\infty(\Omega;\mathbb{R}^{2\times 2})\cap W^{1,2}(\Omega;\mathbb{R}^{2\times 2})$ such that
$$T_{ij,j}+P_i=0, \text{ in } \Omega.$$

Hence, denoting
$$v_{ij}(u)=\frac{u_{i,j}+u_{j,i}}{2}+\frac{u_{m,i}u_{m,j}}{2},$$
we have
$$\frac{u_{i,j}+u_{j,i}}{2}=v_{ij}(u)-\frac{u_{m,i}u_{m,j}}{2},$$
so that, from this, the Gauss-Green formula and the Trace Theorem,
\begin{eqnarray}
J(u)&=&\frac{1}{2}\int_\Omega H_{ijkl}v_{ij}(u)v_{kl}(u)\;dx+\langle T_{ij,j},u_i\rangle_{L^2}-\langle P_i,u_i \rangle_{L^2}-\langle P^t_i,u_i \rangle_{L^2(\Gamma_t)} \nonumber \\ &=&
\frac{1}{2}\int_\Omega H_{ijkl}v_{ij}(u)v_{kl}(u)\;dx-\left\langle T_{ij},\frac{u_{i,j}+u_{j,i}}{2}\right\rangle_{L^2} \nonumber \\ &&+\langle u_i, T_{ij}\nu_j\rangle_{L^2(\Gamma_t)}+\langle (u_0)_i, T_{ij}\nu_j\rangle_{L^2(\Gamma_0)}-\langle P_i,u_i \rangle_{L^2}-\langle P^t_i,u_i \rangle_{L^2(\Gamma_t)} \nonumber \\ &=&\frac{1}{2}\int_\Omega H_{ijkl}v_{ij}(u)v_{kl}(u)\;dx-\left\langle T_{ij},v_{ij}(u)-\frac{u_{mi}\;u_{mj}}{2}\right\rangle_{L^2}\nonumber \\ &&+\langle u_i, T_{ij}\nu_j\rangle_{L^2(\Gamma_t)}+\langle (\hat{u}_0)_i, T_{ij}\nu_j\rangle_{L^2(\Gamma_0)}-\langle P_i,u_i \rangle_{L^2}-\langle P^t_i,u_i \rangle_{L^2(\Gamma_t)} \nonumber \\ &\geq&\frac{1}{2}\int_\Omega H_{ijkl}v_{ij}(u)v_{kl}(u)\;dx-\langle T_{ij},v_{ij}(u)\rangle_{L^2}+\left\langle T_{ij},\frac{u_{mi}\;u_{mj}}{2}\right\rangle_{L^2}\nonumber \\ &&-K_3\sum_{i=1}^3\|u_i\|_{1,4}-K_3\|\hat{u}_0\|_{1,4},\; \forall u \in U,
\end{eqnarray}
for some appropriate $K_3>0$.

From this, the hypotheses on $\{H_{ijkl}\}$ and a Poincar\'{e} type inequality, since $\{T_{ij}\}$ is positive definite, it is clear that $J$ is bounded below so that there exists $\alpha \in \mathbb{R}$ such that
$$\alpha=\inf_{u \in U} J(u).$$

Let $\{u_n\} \subset U$ be a minimizing sequence for $J$, that is, let such a sequence be such that
$$J(u_n) \rightarrow \alpha, \; \text{ as } n \rightarrow \infty.$$

Also from the hypotheses on $\{H_{ijkl}\}$ and a Poincar\'{e} type inequality, we have that there exists $K>0$ such that
$$ \|u_n\|_{1,4} \leq K,\; \forall n \in \mathbb{N}.$$

From the auxiliary result in the last section, there exists $u_0 \in C^0(\overline{\Omega};\mathbb{R}^3) \cap W^{1,4}(\Omega;\mathbb{R}^3)$ such that, up to a not relabeled subsequence,
$$u_n \rightarrow u_0, \text{ strongly in } W^{1,3}(\Omega;\mathbb{R}^3).$$

From such a latter result, up to a not relabeled subsequence, we may obtain
$$\frac{(u_n)_{i,j}+(u_n)_{j,i}}{2}+\frac{(u_n)_{m,i}(u_n)_{m,j}}{2} \rightharpoonup \frac{(u_0)_{i,j}+(u_0)_{j,i}}{2}+\frac{(u_0)_{m,i}(u_0)_{m,j}}{2}, \text{ weakly in } L^{3/2}(\Omega).$$

 Also from the continuity of the Trace operator we get $$u_0= \hat{u}_0, \text{ on } \Gamma_0,$$ so that
 $u_0 \in U.$

Since $L^{3/2}(\Omega)$ is reflexive, from the convexity of $J$ in $\{v_{ij}(u)\}$ and since $\{T_{ij}\}$ is positive definite, we have that
$$\alpha=\liminf_{n \rightarrow \infty}J(u_n) \geq J(u_0),$$
so that,
$$J(u_0)=\min_{u \in U} J(u).$$

The proof is complete.

\end{proof}

\end{document}